\documentclass[11pt]{amsart}

\usepackage{mathtools,amscd,amsthm,amssymb}
\usepackage{graphicx,tikz,xcolor}
\usepackage{tikz-cd}
\usepackage[pagebackref,colorlinks]{hyperref}

\makeatletter
\@namedef{subjclassname@2020}{%
  \textup{2020} Mathematics Subject Classification}
\makeatother

\usepackage[mathscr]{eucal}  

\theoremstyle{plain} 

\newtheorem{theorem}{Theorem}[section]

\newtheorem{THEO}[theorem]{Theorem}
\newtheorem{proposition}[theorem]{Proposition}
\newtheorem{prop}[theorem]{Proposition}
\newtheorem{PROPO}[theorem]{Proposition}
\newtheorem{lemma}[theorem]{Lemma}

\newtheorem{corollary}[theorem]{Corollary}

\newtheorem{problem}[theorem]{Problem}

\theoremstyle{definition}

\newtheorem{definition}[theorem]{Definition}

\newtheorem{remark}[theorem]{\sc Remark}
\newtheorem{rem}[theorem]{\sc Remark}
\newtheorem{example}[theorem]{\sc Example}

\renewcommand{\Im}{\mathop\mathrm{Im}\nolimits}

 \DeclareMathOperator{\Spec}{Spec}
 \DeclareMathOperator{\rank}{rk}
 \DeclareMathOperator{\gr}{gr}

\newcommand{\mf}{\mathfrak}

\newcommand{\K} {\mathbf{k}}

\newcommand {\CC}{\mathscr{C}}
\newcommand {\C}{\CC}
\newcommand{\cA}{{\mathscr{A}}}

\newcommand{\cB}{\mathscr{B}}
\newcommand{\cK} {\mathscr {K}}
\newcommand{\cG} {\mathscr{G}r}

\newcommand{\eps}{\epsilon}

\newcommand{\ve}{\varepsilon }

\newcommand{\HS}{Hilbert\ }
\newcommand{\hs}{\mathscr{H}}
\newcommand{\hsb}{\overline{\hs}}
\newcommand{\md}{\mathrm{md}}
\newcommand{\ds}{\displaystyle}

\begin{document}
\numberwithin{equation}{section}

\title[Deformed graphical algebras] {Deformed graphical zonotopal algebras}

\author[B.~Shapiro]{Boris Shapiro}
\address{Department of Mathematics, Stockholm University, SE-106 91 Stockholm,  Sweden}
\email{shapiro@math.su.se}

\author[I.~Smirnov]{Ilya Smirnov}
\address{BCAM -- Basque Center for Applied Mathematics, Mazarredo 14, 48009 Bilbao, Spain \quad and \quad IKERBASQUE, Basque Foundation for Science, Plaza Euskadi 5, 48009 Bilbao, Spain}
\email{ismirnov@bcamath.org}

\author[A.~Vaintrob]{Arkady Vaintrob}
\address{Department of Mathematics, University of Oregon, Eugene, OR 97403,
 USA}
\email{vaintrob@uoregon.edu }


\begin{abstract}
We study certain filtered deformations of the external zonotopal
algebra of a given graph parametrized by univariate polynomials.
We establish some general properties of these algebras, compute their
Hilbert series for a number of graphs using Macaulay2, and formulate several
conjectures.
\end{abstract}

 \maketitle

\section{Introduction}

\noindent
Let $G$ be a finite  undirected graph. Wagner~\cite{Wa1} and,  independently,
Postnikov and the first author~\cite{PS} introduced a commutative graded
algebra $\CC_G$ whose dimension is equal to the number of spanning
forests of $G$. They also showed that the Hilbert series of $\CC_G$
is a specialization of the Tutte polynomial of $G$ which  enumerates
the spanning forests of $G$ according to
their external activity.
Wagner's initial goal was to construct new algebraic invariants of graphs.
Postnikov and the first author were motivated by the earlier work~\cite{SS,PSS},
where it was shown that for the complete graph $G$, the algebra $\CC_G$ is
isomorphic to the algebra generated by the curvature forms of tautological
Hermitian line bundles on the complete flag manifold. Soon it turned out that these
algebras are connected to several other areas, such as the theory of power
ideals, box splines, enumeration of lattice points, chip firing, etc. They have
been studied under various names: circulation algebras, Postnikov-Shapiro
algebras, forest-counting algebras, and (external) zonotopal algebras. We will
use the latter term reflecting their connection with enumeration of lattice
points in zonotopes (see e.g.~\cite{HR}).

Wagner~\cite{Wa2} and Nenashev~\cite{N1} proved that the algebra $\CC_G$
determines the graphical matroid of $G$. However, non-isomorphic graphs can have
isomorphic algebras. In~\cite{NSh}, Nenashev and the first author introduced a
filtered algebra $\cK_G$ which they called a K-theoretic analogue of $\CC_G$.
They showed that $\cK_G$ and $\CC_G$ are isomorphic as (non-filtered) algebras,
but, unlike $\CC_G$, the filtered algebra $\cK_G$ is a complete invariant of $G$.

The algebra $\cK_G$ is a deformation of $\CC_G$ in the class of filtered
algebras. It is a member of a larger family of filtered deformations $\CC^f_G$
of $\CC_G$ parametrized by polynomials  $f\in \K[u]$ which was introduced
in~\cite{NSh}. In the current work, we begin to study this family of algebras
trying to understand their relationship to each other and to the graph $G$.

We begin  Section~\ref{sec:prelim} with a review of the definition and
properties of the graded algebra $\CC_G$. Then in Section~\ref{sec:alg} we turn
to the deformed algebras  $\CC^f_G$ and establish some general facts about them.
In particular, we prove that under some mild
nondegeneracy assumption on $f$, the algebra  $\CC^f_G$ is isomorphic
to $\CC_G$ as an unfiltered algebra.
We introduce a natural stratification of the space of such algebras
for a given graph $G$.
Section \ref{sec:prel} discusses several properties of deformed
zonotopal algebras and  the latter stratification.
Unfortunately,  unlike the graded case, we cannot  explicitly find the
Hilbert  sequences of these  algebras. In Section~\ref{sec:exp} we
collect  our computations of the Hilbert sequences for several
examples performed in Macaulay 2.
Finally, in Section~\ref{sec:out} we present a number of conjectures
and questions for further study.%

\subsection*{Acknowledgments}
The first author was supported by the grant 2021-04900 of the Swedish
Research Council.
The second author was supported by a fellowship from ``la Caixa''
Foundation (ID 100010434), fellowship code LCF/BQ/PI21/11830033, and from the European Union’s Horizon 2020
research and innovation programme under the Marie
Skłodowska-Curie grant agreement No 847648.

\section{Preliminaries} \label{sec:prelim}

\subsection{Graded zonotopal algebras}
\label{sec:graded-algebras}

By a \emph{graph} we understand a finite undirected multigraph $G=(V,E)$,
possibly with loops, with a vertex set $V$ and a (multi)set of edges $E$.

Graphs form a category $\cG$ with morphisms $G=(V,E) \to G'=(V',E')$  defined as
maps of pairs
\begin{equation}
  \label{eq:categ}
(\gamma,\lambda) \colon (V,E)\to (V',E'),
\end{equation}
\emph{injective on edges} and preserving incidences between vertices and edges
(i.e.\ if $e\in E$ is an edge in $G$ connecting $u$ and $v$, then $\lambda(e)$
is an edge  in $G'$ connecting $\gamma(u)$ and $\gamma(v)$).
\smallskip

There are two special kinds of graph morphisms,
\emph{edge deletions} $j_e\colon G-e\to G$
and \emph{edge contractions} $\pi_e\colon G\to G_e$, where $e\in E$ is an edge of $G$,
$G-e$ is the subgraph of $G$ obtained by removing $e$ from $E$, and $G_e$ is
the graph obtained from $G$ by identifying the endpoints $u$ and $v$ of $e$ (and
thus creating a loop for each edge connecting $u$ and $v$).
\smallskip

All algebras in this paper are commutative unital algebras over
a fixed field $\K$ of characteristic $0$.

\begin{definition}
 Given a graph $G=(V,E)$, its \emph{edge algebra} $\Phi_G$ is the
 quotient of the polynomial algebra in \emph{edge variables}
 $\phi_e, e\in E$, by their squares,
    \begin{equation}
    \label{eq:sq-free}
    \Phi_G:= \K[\phi_e: e\in E]/(\phi_e^2).
  \end{equation}
\end{definition}
The edge algebra is a local algebra of dimension $2^{|E|}$ isomorphic to the
tensor product of $|E|$ copies of the algebra of dual numbers
$\K[\ve]/(\ve^2)$, namely
$ \ds\Phi_G\simeq\bigotimes_{e\in E} \big(\K[\phi_e]/(\phi_e^2)\big).$
It inherits a standard grading from the polynomial algebra. A basis of the $k$th
graded component is given by square-free words of length $k$ in edge variables
$\phi_e$.

Observe that this construction is functorial. Every graph morphism~\eqref{eq:categ} induces a
natural homomorphism
of algebras $\lambda^*\colon \Phi_{G'}\to \Phi_{G}$, defined on generators
$\phi_{e'}\in \Phi_{G'}$
as
\begin{equation}
  \label{eq:functor}
\lambda^*(\phi_{e'}):=\sum_{e\in \lambda^{-1}(e')}\phi_e,
\end{equation}
i.e., $\lambda^*(\phi_{\lambda(e)})=\phi_e$ and
$\lambda^*(\phi_{e'})= 0$, if $e'\in E'$ is not in the image of $\lambda$.

\begin{definition}
Let $G=(V,E)$ be a graph with a linear order $<$ on its vertex set $V$.
The \emph{zonotopal algebra} of $G$ is the subalgebra $\CC_G$ of the edge algebra $\Phi_G$ generated by the  elements:
\begin{equation}\label{eq:gens}
X_v =\sum_{e\in G} c_{v,e} \phi_e,\quad v\in V,
\end{equation}
 called \emph{vertex flows} where
\begin{equation}\label{eq:def}
  c_{v,e}=\begin{cases} \;\;\;1\quad \text{if}\;
           e=\{v,u\}, \ v<u,\\   -1\quad\text{if}\; e=\{v,u\}, \ v>u,\\
                     \;\;\;0\quad \text{otherwise}.
\end{cases}
\end{equation}
\end{definition}

\begin{remark}
Even though the coefficients $c_{v,e}$ in~\eqref{eq:def}, and thus the
elements~\eqref{eq:gens}, depend on the chosen ordering $<$ of the vertex set $V$, the
subalgebras obtained from two different orderings will be identified under a
graded automorphism of the algebra $\Phi_G$ given by changing signs of some
generators $\phi_e$.
It is clear that loop edges of $G$ do not contribute to the flow
generators $X_v$. For this reason, in earlier papers~\cite{PS,N1,NSh} on this
topic, graphs with loops were not considered.
In this work we allow loops because they naturally appear when we consider
graph homomorphisms involving edge contractions.
 \end{remark}

One of the important properties of algebra $\CC_G$ is that it is
functorial with respect to graph
morphisms~\eqref{eq:categ}.
\begin{prop}
For  every graph morphism $(\gamma,\lambda)\colon G\to G'$ is , the homomorphism
$\lambda^*$~\eqref{eq:functor} sends the  subalgebra $\CC_{G'}\subset \Phi_{G'}$
to $\CC_G\subset \Phi_G$.
\end{prop}

\begin{proof}
All we need to show is that for every vertex $v'\in V'$,
  the image of the generator $X_{v'}$ of $\CC_{G'}$ under the  homomorphism
  $\lambda^*$ belongs to $\CC_{G}\subset \Phi_G$.
  Indeed, if $v'\not \in \gamma(V) $ then $\lambda(E)$ contains no
  edges incident to $v'$, and so $\lambda^*(X_{v'})=0$. If
  $v'\in \gamma(V)$, then, because of the injectivity of $\lambda$, we
  have $\ds \lambda^*(X_{v'})=\sum_{v\in\gamma^{-1}(v')}X_v\in \CC_G$.
\end{proof}

\begin{remark}
  (1) The definition of $\CC_G$ involves choosing a linear order $<$ on $V$. In
  the above proof we assume that the orders $<$ and $<^\prime$ on $V$ and $V'$ resp. are compatible,
  i.e.\ they are chosen in such a way that the map $\gamma\colon V\to V'$ is monotone.
  \\[3pt]
  (2) Since the loop edges of $G$ do not contribute to the generators $X_v$ of
  the algebra $\CC_G$, we can remove loops from $G$ without affecting $\CC_G$.
  In particular,  let $G/e$ be the graph obtained from $G$ by contracting an edge
  $e\in E$ by removing $e$ (without creating a new loop) and identifying the
  endpoints $u$ and $v$ of $e$. Then the contracting morphism
  $\pi_e\colon G\to G_e$ gives an injective homomorphism
  \begin{equation}
  \pi_e\colon \CC_{G/e}=\CC_{G_e}\to \CC_G
      \end{equation}
  which sends the generator $X_w$ corresponding to the new vertex $w=\{u,v\}$ to
  $X_v+X_u\in \CC_G$.
\end{remark}

As was proved in~\cite {Wa2, PSS}, the algebra $\CC_G$ is a power algebra,
i.e.\ it is isomorphic to a quotient of the polynomial algebra by an ideal
generated by powers of linear forms.
Namely, for a subset $I\subset V$ of vertices denote by $D_I$ the number
of edges $e\in E$ connecting a vertex from $I$ with one in the complementary
subset $I^c=V-I$.
\begin{theorem}[\cite{Wa2,PSS}]
  \label{thm:power}
The zonotopal algebra $\CC_G$ is isomorphic to the quotient of the polynomial
algebra $\K[x_v\, : \, v\in V]$ by the ideal $(p_I)_{I\subseteq V } $
generated by the polynomials
  \begin{equation}
    \label{eq:relF}
    p_I=\Bigl( \sum_{v\in I}x_v\Bigr)^{D_I+1}.
  \end{equation}
\end{theorem}

The dimension and the Hilbert function of the algebra $\CC_G$ were also found
in~\cite{Wa2,PSS}.

\begin{theorem}
\label{th:forests}
{\rm (i)}  The dimension of the algebra $\CC_G$ is equal to the number of
spanning subforests in $G$ (which is the same as the number of acyclic
subsets of edges $S\subseteq E$).

{\rm (ii)} The dimension of the $k$th graded component of $\CC_G$ is
equal to the number of subforests $S\subseteq E$ with the external
activity $|E-S|-k$.
\end{theorem}

\begin{proof}
  \label{rem:del-contr}
Functoriality of the algebra $\CC_G$ with respect to graph homomorphisms
leads to simple proofs of Theorems~\ref{thm:power} and \ref{th:forests}.

Indeed, let $e\in E$ be a non-loop edge of a graph $G$.
Consider two algebra homomorphisms, a projection $j_e^*\colon \CC_G\to \CC_{G-e}$
(which corresponds to sending $\phi_e\in \Phi_G$ to $0$) and an embedding
$\pi_e^*\colon \CC_{G/e}=\CC_{G_e}\to \CC_G$ (which maps $X_w$ to $X_u+X_v$).

Denote by $\ds \delta_e: = \frac{\partial}{\partial
  \phi_e}\colon \Phi_G\to \Phi_{G-e}$, the partial derivative with respect to the
edge variable $\phi_e\in \Phi_G$. Modulo $(\phi_e)$ the map $ \delta_e$ is indeed a derivation
of the edge algebra $\Phi_G$ which sends the subalgebra $\CC_G$ onto
$\CC_{G-e}$ and generates an exact sequence of graded spaces
$$
0\to \CC_{G/e}\to \CC_G \to \CC_{G-e}[1]\to 0,
$$
where the rightmost map (induced by $\delta$) decreases the grading  by $1$.
This exact sequence implies the relation 
\begin{equation}\label{eq:Tutte}
H_{\CC_G}(t)=H_{\CC/e}(t)+tH_{\CC_{G-e}}(t),
\end{equation}
for the \HS series which proves Theorem~\ref{th:forests} by induction on $|E|$.

To settle Theorem~\ref{thm:power} we can use similar maps and an exact sequence
$$
0\to \cB_{G/e}\to \cB_G \to \cB_{G-e}[1]\to 0,
$$
for quotient algebras $\cB_G:=\K[x_v\, : \, v\in V]/(p_I)_{I\subseteq V } $,
where the polynomial $p_I$ is given by~\eqref{eq:relF}.
\end{proof}

\

In~\cite{Ne} G.~ Nenashev has shown that $\CC_G$ contains all
information about the graphical matroid of $G$ and only it.

\begin{PROPO}[Theorem~5 of \cite{Ne}]\label{nenashev} Given two undirected
(multi)graphs $G_1$ and $G_2,$ algebras $\CC_{G_1}$ and $\CC_{G_2}$ are
isomorphic  if and only if the graphical matroids of $G_1$ and $G_2$ coincide.
(The  latter isomorphism can be thought of either as graded or as non-graded,
the statement holds in both cases.)
\end{PROPO}

\subsection{Deformed  algebras $\CC^f_G$}
\label{sec:alg}

\subsubsection{Definition}
\label{sec:definitions}

The main object of our study is a certain family of filtered algebras, introduced
in \cite{NSh}, which we call \emph{deformed zonotopal algebras}.
\begin{definition}
A formal power series $f=a_0+a_1u+a_2u^2+\ldots \in \K[[u]]$ is called
 \emph{nondegenerate} if $a_1\ne 0$, i.e.\ when $f'(0)\ne 0$.
\end{definition}

\begin{definition}
For a graph $G=(V,E)$ and a nondegenerate power series $f\in \K[[u]]$,  we define
the \emph{deformed zonotopal algebra} of $G$ associated to $f$ as the
subalgebra of the edge algebra $\Phi_G$   generated by the elements
$$Y_v:=f(X_v)=f\bigl(\sum_{e\in E} c_{v,e}\phi_e\bigr), \ v\in V,$$
where $c_{v,e}$ are given by~\eqref{eq:def}.
\end{definition}
In particular, for $f=u$,
the algebra  $\CC^f_G$ is the usual zonotopal algebra
 $\CC_G$ discussed above and for  $f=e^u$,   this algebra
 coincides with the $K$-theoretic analog $\cK_G$ of $\CC_G$ studied in \cite{NSh}.

\begin{rem}
\label{rem:nilpot}
Since $\phi_e^2=0$, the element $X_v$ is nilpotent with $X_v^n=0$
for $n>\deg v$, and so plugging it into a power series is well-defined.
Moreover, this argument also shows that the terms of $f$ of degree higher than
$$\md_G:=\max_{v\in V} \deg v,$$
the maximal degree of a vertex in $G$,
do not affect any of the generators $Y_v$ of $\CC_G^f$. Therefore, we can
restrict our attention to those $f\in \K[[u]]$ which are polynomials in $u$ of
degree at most $\md_G$.
\end{rem}

The algebra $\CC^f_G$ is endowed with an increasing filtration
\begin{equation}
  \label{eq:filtration}
\K=\CC^{f,0}_G\subset\CC^{f,1}_G\subset\CC^{f,2}_G\subset\ldots,
\end{equation}
where the subspace $\CC^{f,k}_G$ is spanned by the monomials of degree at most
$k$ in the generators $Y_v=f(X_v)$.

\begin{rem}
\label{rem:const}
Notice that neither the algebra $\CC^f_G$ nor this filtration depend
on the constant term $a_0=f(0)$ of $f$, since changing $f$ by a constant
modifies $Y_v$ by this constant which is an element of $\K=\CC^{f,0}_G$.
It is also clear that multiplying $f$ by a nonzero constant does not change
the filtration~\eqref{eq:filtration}.

For this reason, from now on,  we will assume that $f$ has no constant term
i.e.\ that $f-u^k\in u^{k+1}\K[[u]]$ for $k\ge 1$.
\end{rem}

When $f = u$, this filtration coincides with the filtration induced by the
grading on $\CC_G$.
Let us now present some basic properties of $\CC^f_G$ proven in \cite{NSh}.

\begin{PROPO}[{\cite[Proposition~2]{NSh}}]
\label{prop2}
If $f$ is a nondegenerate series, then the algebra $\CC^f_G$  coincides, as a
subalgebras of $\Phi_G$, with the usual zonotopal algebra $\CC_G$.
In other words, in this case the only difference between the graded algebra
$\CC_G$ and  $\CC^f_G$ is in their filtrations.
\end{PROPO}

\begin{proof} Firstly, by Remark~\ref{rem:nilpot} we can assume that $f$ is a
polynomial.
 Secondly, if $f$ is a polynomial we have the inclusion  $\CC_G^f\subseteq \CC_G$
 because $f(X_i)\in \CC_G$ for every $i$.
 Finally, there   exists a polynomial $g(u)$ such that $X_i=g(f(X_i))$ for every $i$
 which finishes the proof.
 Indeed,  the formal power series
 $g(u)=\sum_{i \geq 1} a_i u^i$ can be  found by requiring that
 $$
 u = f(\sum_{i \geq 1} a_i u^i) = c_1a_1 u + (c_1a_2 +
  c_2a_1^2) u^2 + (c_1a_3 + 2c_2 a_1a_2 + c_3a_1^3) u^3
  + \cdots.
  $$
Under the assumption that $c_1\neq 0$, this system of equations can solved for
each $a_j$ consecutively by induction.
The resulting power series $g(u)$
can be truncated to a polynomial in view of
Remark~\ref{rem:nilpot}.
\end{proof}

This proposition together and the above remark explain  why we focus our
attention on (nondegenerate) polynomials  $f = u + \ldots$.

\begin{THEO}[{\cite[Theorem~6]{NSh}}]
\label{th:Isom-f}
Let $f$ be a polynomial with non-vanishing linear and quadratic terms and let
$G_1$ and $G_2$ be two simple graphs without isolated vertices.
Then  $\CC^f_{G_1}$ and $\CC^f_{G_2}$ are isomorphic as filtered algebras  if
and  only if  the graphs $G_1$ and $G_2$ are isomorphic.
\end{THEO}

\medskip

\begin{definition}
  \label{def:param_space}
Given a graph $G$, we will call the affine space
\begin{equation}
  \label{eq:poly-space}
\cA_G=\{f\in \K[u] \,|\, f(0)=0, f'(0)=1, \deg f\leq \md_G \}
\end{equation}
\emph{the space of parameters of deformed zonotopal algebras} of $G$.
\end{definition}

For $f\in \cA_G$, we will be interested in the  \emph{\HS  sequence}
 $$\hs^f_{G} := (\dim_\K \CC^{f,j}_G/\CC^{f,j-1}_G)_{j \geq 0}$$
 of the filtered algebra $\CC^f_G$ (where, by convention, $\CC_G^{f,-1}=0$).

 \begin{PROPO}[{\cite[Proposition~3]{NSh}}; see also Theorem~\ref{thm semi}.]
 \label{pr-hil}
There exists a non-empty Zariski open subset $U\subset \cA_G$ such that the
\HS sequences $\hs^f_G$, for all $f\in U$, are the same and  maximal among all possible \HS sequences $\hs^f_{G}$,  $f\in \cA_G$  in the
lexicographic order.
\end{PROPO}
We will call the above maximal \HS sequence the \emph{general \HS sequence} and will denote it simply by $\hs_G$.

\section{Algebraic properties of deformed zonotopal algebras}
\label{sec:prel}

\subsection{Generators and relations}
\label{sec:gen-rels}

\begin{proposition} \label{prop:general}
Let $A$ be a finite dimensional local algebra over $\K$ with maximal ideal
$m$ and with a set of algebra generators $x_1,\dots , x_n\in m$.
Let $f\in \K[[u]]$ be a nondegenerate series with $f(0)=0$. Then 
\begin{enumerate}
\item the map $f\colon A \to A, \ a\mapsto f(a)$ is well-defined
  and invertible;
\item the elements $y_1 = f(x_1), \dots, y_n = f(x_n)$ generate $A$;
\item for $L\in \K[t_1,\ldots,t_n]$, the relation $L (x_1, \ldots, x_n) = 0$
holds in $A$ if and only if  the relation $L(f^{-1}(y_1),\dots, f^{-1}(y_n)) = 0$
holds in $A$.
\end{enumerate}
\end{proposition}
\begin{proof}
To settle (1) notice that since $A$ is an Artinian algebra, any $a\in A$ is nilpotent
which implies that $f(a)$ is a finite sum. Therefore $f(a)$ is well-defined for any
$a$. The inverse of $f$ has been already constructed in the proof of
Proposition~\ref{prop2}.

Now, (2) is obvious since $f$ is invertible as a map of $A$ and therefore
$x_i=f^{-1}(y_i)$ for every $i$. The first part of (3), claiming  that the
relation $$L(f^{-1}(y_1),\dots, f^{-1}(y_N))=0$$ holds in $A$, is  obvious.
Conversely, assume that a relation $R(y_1,\dots, y_n)$ $=0$  holds in $A$. Then the
equation $R(f(x_1),\dots, f(x_n))=0$ gives a relation in $A$ in terms of the original
generating set $x_1,\dots, x_n$ which then satisfies the claim  because of invertibility of $f$.
\end{proof}

\begin{corollary}\label{cor:relations}
For a graph $G=(V,E)$ and a non-degenerate $f \in \K[[u]]$, we have the isomorphism of $\K$-algebras 
$\CC^f_G \cong \K[y_v\, : \, v\in V]/I_G^f$; the latter is the quotient of the polynomial algebra by the ideal $I_G^f$ generated by two sets
$\{y_v^{\deg v+1} \mid  v\in V\}$
and
$\{  \left(\sum_{v\in  I}f^{-1}(y_v)\right)^{D_I+1}\mid I \subset V, |I| \geq 2 \}$
where $D_I$ has been defined before Theorem~\ref{thm:power}.
 \end{corollary}
\begin{proof}
 As we have already seen, $X_v^{\deg v + 1} = 0$ which implies that
$Y_v^{\deg v + 1} = (f(X_v))^{\deg v + 1} = 0$
by plugging directly into the formula for $f(u)$.
By Proposition~\ref{prop:general} and Theorem~\ref{thm:power},  
$\left(\sum_{v\in I}f^{-1}(Y_v)\right)^{D_I+1}$ are still relations in $\CC^f_G$. 
Thus by mapping $y_i\mapsto Y_i$ we see that  $\CC^f_G$ is 
a homomorphic image of $\K[y_v\, : \, v\in V]/I_G^f$. 
Due to invertibility of $f(u)$, there are no other relations (i.e., the kernel of the homomorphism is zero) as this would contradict Theorem~\ref{thm:power}.
\end{proof}

In view of Proposition~\ref{prop2}, we can view each $f$ as a certain choice of filtration
on a fixed algebra $\CC_G$. We can always recover the original algebra by taking
the associated graded ring with respect to the distinguished ideal.

\begin{corollary}\label{cor: ass graded}
Let $f$ be a nondegenerate polynomial and let $J$ be the ideal
of the algebra $\CC_G^f$ generated by the elements $Y_v=f(X_v)$, $v\in V$.
Then the associated graded algebra $\mathrm{Gr}_J(\CC_G^f)$ for the $J$-adic filtration is isomorphic
to the graded algebra $\CC_G$. In particular, the Hilbert series of $\CC^f_G$ and
$\CC_G$ coincide.
\end{corollary}
\begin{proof}
In the notation of Corollary~\ref{cor:relations}, the associated graded
algebra  $\CC^f_G$ is the quotient of $\K[Y_v]$ by the initial ideal of $I_G^f$,
 i.e., the ideal generated by the lowest degree homogeneous forms of elements
$f \in I_G^f$. Clearly, $Y_v^{\deg v+1}, \ v\in V,$
are contained in the initial ideal and the expression
      $\left(\sum_{v\in I}f^{-1}(Y_v)\right)^{D_I+1}$
will contribute the homogeneous form arising from the linear part of $f^{-1}$,
i.e., $\left(\sum_{v\in I} Y_v\right)^{D_I+1}$.
These forms are the defining relations of the algebra $\CC_G$ and there are no other
relations because $\dim \gr (\CC^f_G) =\dim \CC^f_G = \dim \CC_G$ by Proposition~\ref{prop2} and properties of associated graded algebras.
Thus the graded algebras $\mathrm{Gr}_J(\CC_G^f)$ and $\CC_G$ are isomorphic.
\end{proof}

Several concrete examples of such relations can be found later in the text.

\begin{remark}
The relation $\ds\sum_{v\in V}f^{-1}(Y_v)=0$ corresponding to $I=V$
in Corollary~\ref{cor:relations}
can be used to remove half of the   
generators of the ideal $I_G^f$.
Namely, for a subset $I\subset V$ with $1 <|I|< |V| $, consider the
complementary subset $I^c=V- I$. If $|I|< |I^c|$, then keep the generator
$\left(\sum_{v\in I}f^{-1}(Y_v)\right)^{D_I+1}$ corresponding to $I$ and remove the one corresponding to $I^c$.
If $|I^c|> |I|$, then keep the generator corresponding to $I^c$
and remove the one corresponding to $I$. Finally, if $|I|=|I^c|$ (which can
happen only if $|V|$ is even), then keep any of these two generators and remove
the other one. The generator corresponding to $I=V$ should not be removed.
\end{remark}

\subsection {Stratification of the space of deformed zonotopal algebras of a graph}\label{sec:funct}
Let $\cA_G$ be the affine space~\eqref{eq:poly-space} of parameters of algebras
$\CC_G^f$ for a graph $G$.
Our ambition is to study the stratification of $\cA_G$ according to the \HS
sequences $\hs_G^f$ of $\CC^f_G$.

Since the generators $Y_v, \ v\in V,$ of the algebra $\CC^f_G$ are nilpotent,
the corresponding \HS sequence $\hs_G^f$ has finitely many non-zero terms.
We denote by $\hsb_G^f$ the finite sequence obtained by removing all zero terms
of $\hs_G^f$.
\begin{definition}
Given a graph $G$ and a sequence $M=(1,m_1,m_2,\dots, m_n)$ of positive integers,
the {\it associated \HS  stratum} is the subset $S_G^M$ of the parameter space $\cA_G$
consisting of all $f$ such that $\hsb_G^f= M$.
\end{definition}
We will show that each $S_G^M$ is a constructible algebraic set and we will be
interested in the decomposition of its closure into irreducible closed algebraic
components. (Observe that  $S_G^{M}$ might be empty.)
   Let us describe the adjacency of these strata.

\subsection{Semicontinuity}
First recall the following notion.

\begin{definition}
 Let $X$ be a topological space and let
 $(\Lambda, \prec)$ be a partially ordered set.
We say that a function $g\colon X \to \Lambda$ is upper semicontinuous if for
every $\lambda \in \Lambda,$ the set
\[
g^{-1} (\prec \lambda) := \{x \in X \mid g(x) \prec \lambda\}
\]
is open.
\end{definition}

\begin{lemma}\label{l linear algebra semi}
Let $A$ be a commutative ring and $M$ be a matrix with entries in $A$.
For a prime $\mf p \in A$, let $M(\mf p)$ be the matrix obtained from $M$ by
replacing the entries by their images in $k(\mf p)$.
Then the real-valued function $\mf p \mapsto \rank M(\mf p)$ (respectively,
$\mf p \mapsto \dim \ker M(\mf p)$) is  lower (resp., upper) semicontinuous.
\end{lemma}
\begin{proof}
We use the fact that non-vanishing of a minor is an open condition.
Namely, if $B$ is a square matrix then $\det B(\mf p) \neq 0$ if and only if
$\det B \notin \mf p$.
Hence, for any $\mf p$ such that $\rank M(\mf p) \geq r$, there is an open
neighborhood where the same condition holds.
\end{proof}

\begin{theorem}\label{thm semi}
Let $A$ be a commutative ring and $R=A[X_1, \ldots, X_N]/I$
be a finite $A$-module. Then
\begin{enumerate}
\item the function $\mf p \mapsto \dim_{k(\mf p)} R \otimes_A k(\mf p)$ is upper
  semicontinuous on $\Spec R$;
\item  for any positive integer $m$,
 $H_m \colon \mf p \mapsto \dim_{k(\mf p)}
  \langle X_1^{\alpha_1}\cdots X_N^{\alpha_N} \mid \alpha_1 + \cdots + \alpha_n
  \leq m  \rangle$,
where the latter is a submodule of $R \otimes_A k(\mf p)$,
is also a lower semicontinuous function on $\Spec A$;
\item for $\mf p \subset \mf q,$ we have the equality $H_m(\mf p) = H_m(\mf q)$
  if and only if
$\langle X_1^{\alpha_1}\cdots X_N^{\alpha_N} \mid \alpha_1 + \cdots + \alpha_n
\leq m  \rangle \otimes_A A/\mf p_\mf q$
is a free $\otimes_A A/\mf p_\mf q$-module.
\end{enumerate}
\end{theorem}
\begin{proof}
Since $R$ is a finite $A$-module, it can be generated (as a module) by monomials
of bounded degree. We fix such system of generators of $R$ to define its
presentation as an $A$-module in the form:
\[
A^{\oplus m} \xrightarrow{M} A^{\oplus n} \to R \to 0.
\]
It is now clear that
$\dim_{k(\mf p)} R \otimes_A k(\mf p) = n - \rank M(\mf p)$,
so the first claim follows from Lemma~\ref{l linear algebra semi}.

For the second claim, we take $V_k$ to be a free submodule of $A^{\oplus n}$
corresponding to monomials of degree at most $k$.
Then
\[
\dim_{k(\mf p)} \langle X_1^{\alpha_1}\cdots X_N^{\alpha_N} \mid \alpha_1 +
\cdots + \alpha_n \leq k  \rangle
=
\rank_A
V_k - \dim_{k(\mf p)} \Im M(\mf p) \cap V_k(\mf p).
\]
By a standard linear algebra argument
$$
\dim_{k(\mf p)} (\Im M(\mf p)) \cap V_k(\mf p) = \dim_{k(\mf p)} \ker [M \mid
-V_k](\mf p)
$$
which is an upper semicontinuous function by Lemma~\ref{l linear algebra semi}.

Last, observe that for any finite $A$-module $M$, the equality
$\dim_{k(\mf p)} M \otimes_A k(\mf p)= \dim_{k(\mf q)} M \otimes_A k(\mf q)$
is equivalent to $M \otimes_A A/\mf p_\mf q$
being $A/\mf p_\mf q$-free since the generic rank and the minimal number
of generators of this module have to be equal.
\end{proof}

\begin{corollary}
Let $A$ be a commutative ring and $R=A[X_1, \ldots, X_N]/I$
be a finite $A$-module. Suppose that $R$ is generated (as a module) by monomials
of degree at most $D$ and set $\Lambda = \mathbb Z^{\oplus D + 1}$ with the
lexicographic order.
In the notation of Theorem~\ref{thm semi}, define the function
$H\colon \Spec A \to \Lambda, \mf p \mapsto (H_0, H_1, \ldots, H_D)$.
Then $H$ is lower semicontinuous.
Moreover, for any vector $\lambda = (\lambda_0, \lambda_1, \ldots)$,  the set
\[
H^{-1} (\lambda) = \{\mf p \in \Spec A \mid H(\mf p ) = \lambda\}
\]
is the intersection of an open set $H^{-1} (\succeq_{lex} \lambda)$ and a closed
set $H^{-1} (\preceq_{lex} \lambda)$.

In particular, the closure of $H^{-1} (\lambda)$ coincides with
$H^{-1} (\preceq_{lex} \lambda)$.
\end{corollary}
\begin{proof}
For the next claim, we first note that we may stop at $H_D$.
Since each function $H_k$ is discrete, the set
 $H_k^{-1} (\succeq a) = H_k^{-1} (\succ  a - 1)$ is open.
Thus if we set $\lambda = (\lambda_0, \ldots, \lambda_D)$, we may decompose
\begin{align*}
H^{-1} (\succ_{lex} \lambda) = &H_0^{-1} (\succ  \lambda_0) \bigcup H_0^{-1}
   (\succeq  \lambda_0) \cap H_1^{-1} (\succ \lambda_1) \bigcup \\
&H_0^{-1} (\succeq  \lambda_0) \cap H_1^{-1} (\succeq  \lambda_1) \cap H_2^{-1}
  (\succ  \lambda_2) \bigcup \cdots \bigcup \\
&H_0^{-1} (\succeq  \lambda_0) \cap \cdots \cap H_{D - 1}^{-1} (\succeq
  \lambda_{D-1}) \cap H_{D}^{-1} (\succ  \lambda_{D})
\end{align*}
and see that $H^{-1} (\succ_{lex} \lambda) $ is open.
Similarly,  replacing in the last line $H_{D}^{-1} (\succ  \lambda_{D})$ with
$H_{D}^{-1} (\succeq  \lambda_{D})$ we see that $H^{-1} (\succeq_{lex} \lambda)$
is also open.
It remains to note that $H^{-1} (\preceq_{lex} \lambda)$ is closed because it
can be decomposed as:
\begin{align*}
  H^{-1} (\preceq_{lex} \lambda) = &H_0^{-1} (\prec \lambda_0) \bigcup
   H_0^{-1}(\preceq  \lambda_0) \cap H_1^{-1} (\prec   \lambda_1) \bigcup \cdots\\
       & \bigcup H_0^{-1} (\preceq  \lambda_0) \cap \cdots \cap
        H_{D - 1}^{-1} (\preceq  \lambda_{D-1}) \cap H_{D}^{-1} (\prec  \lambda_{D}).
\end{align*}
\end{proof}

\begin{corollary}\label{cor constructive}
The stratum $S_G^{M}$ is constructive and, in the lexicographic partial order,
we have $\overline{S_G^{M}} = S_G^{\preceq M}$.
\end{corollary}
\begin{proof}
  By Remark~\ref{rem:nilpot},
  polynomial in $\cA_G$ is a specialization of the  generic polynomial
  $f = u +  T_2u^2 + \cdots + T_du^d$, where $d=\md_G$ is the
  maximal degree of   $G$.
For this $f$, we define the algebra $R = \CC^{[f]}_G$ as the
quotient of
$\K[X_v, T_2, \ldots, T_n]$, $v\in V$, by
the relations from Corollary~\ref{cor:relations}. Note that the obvious relation  
$f^{-1}( f(u)) = u$ gives explicit polynomial formulas for the coefficients
of $f^{-1}$ in terms of $T_i$.
Since $R$ is a finitely generated module over
$A = \K[T_2, \ldots, T_d]$,
the theorem applies and we may use
the fact that
$\cA_G$ can be identified with the set
of $\K$-rational points of $\Spec A$.
\end{proof}

Next we introduce a natural $\K^\star$-action on $\cA_G$.

\begin{lemma}\label{lm:action} For any
graph
$G$, the natural $\K^\star$-action  on $\cA_G$ given by
$f(u)\mapsto  \frac{1}{\epsilon}f(\epsilon u)$ with
$\epsilon \in \K\setminus\{0\}$,
 preserves the  Hilbert stratification.
 This action is free on  $\cA_G\setminus \{u\}$ and
the point $u\in \cA_G$  belongs to
the closure of every Hilbert stratum.
\end{lemma}

\begin{proof}
  We substitute
$Y_v = X_v/\epsilon$
  to get an isomorphism of algebras
\[
\K[X_v \,:\, v\in V]/I_G^{f(u)} \cong \K[Y_v \,:\, v\in V]/(I_G^{f(\epsilon u)}).
\]
The second claim follows
from the fact that if $f=u+\dots \ne u$,
then $f(u)\neq  \frac{1}{\epsilon}f(\epsilon u)$ for $\eps\neq 1$.
But, by Borel's fixed point theorem, the closure of each stratum should contain
a fixed point.
\end{proof}

Denote by $P\cA_G=\left(\cA_G\setminus \{u\}\right)/\K^\star$ the weighted
projective space obtained as the latter quotient. Taking the quotient we  obtain the induced Hilbert
stratification of $P\cA_G$ which we will be interested in. Consider the (finite)
poset of Hilbert strata. This poset has the minimal element corresponding to
$\{u\}$ and the maximal element corresponding to the stratum with a generic \HS
sequence.

\begin{corollary}
  The \HS sequence of the graded algebra $\CC_G  $ is minimal.
The maximal length of the \HS sequence of a generalized zonotopal algebras in
$\cA_G$ is attained for $f(u)=u$ and equals $|G|$ which is the total number of
edges in $G$.
\end{corollary}

\begin{example}
For $G=K_4$, $P\cA_{K_4}$ is a weighted projective line (topologically $S^2$).
Its Hilbert stratification consists of two points corresponding to $f(u)=u+u^2$
and $f(u)=u+u^3$ and a complex $1$-dimensional stratum (coinciding with $S^2$
minus two points) which is the factor of $u+au^2+cu^3$ with $b\neq 0$ and $c\neq
0$ mod the above $\K^\star$-action, see subsection~\ref{K4} below.
\end{example}

\subsection{Specialization}

If $H$ is a subgraph of $G$ and $v$ is a vertex of $H$, then
the embedding $H\hookrightarrow G$ sends the generator $X_v^G$ of $\CC_G^f$ to
$X_v^H$. Therefore, we obtain a surjective homomorphism
$\pi_H \colon \C_G^f\to \C_H^f$.

\smallskip
We will now study the effect of this map on stratifications.

\begin{lemma}\label{lm: direct criterion}
Let $A$ be a commutative ring and $M$ be an $A$-module
such that $M = M_1 \oplus M_2$. Suppose that $0 \neq N \subset M$ is a submodule
that satisfies a commutative diagram
\[\begin{tikzcd}
 N \arrow{r}{f} \arrow[hookrightarrow, swap]{d}{} & N_2
 \arrow[hookrightarrow]{d}{}
 \arrow{r}{}  & 0  \\
M\arrow{r}{\pi} & M_2 \arrow{r}{}  & 0
\end{tikzcd}
\]
for some $A$-module $N_2 \neq 0, N$. Then $N = M_1 \cap N \oplus M_2 \cap N$.
\end{lemma}
\begin{proof}
It is easy to check that the diagram implies that $\ker f = N \cap M_1 \neq 0$.
Then the inclusion $N_2 \subseteq M_2$ given in the diagram implies that
$N \subseteq N\cap M_1 + M_2$ as submodules of $M$. Since $N$ is a submodule,
this forces the containment $N \subseteq N\cap M_1 + N \cap M_2$ and the claim
easily follows.
\end{proof}

As in the proof of Corollary~\ref{cor constructive},
we can define $\C_G^f(A)$ for any ground ring $A$ and $f \in A[u]$.
As above, we have a surjective homomorphism of  $A$-algebras,
$\C_G^f(A) \to \C_H^f(A)$.

\begin{theorem}\label{thm graph splitting}
Let $A$ be a commutative ring and $f \in A[u]$ such that
$f \in (u)\setminus (u^2)$.
Let $H \subset G$ be graphs.
Then the natural projection map splits and allows to identify
$\C_H^f(A)$ with a direct summand of $\C_G^f(A)$
in the category of  filtered $A$-modules.
\end{theorem}
\begin{proof}
It suffices to assume that $H$ is obtained by removing a single edge.
We can similarly define the algebras $\Phi_G(A), \Phi_H(A)$ as square-free
algebras on the sets of edges with coefficients in $A$.
It is clear from the relations that $\C_G^f(A) \subset \Phi_G(A)$.

We note that the map
$\pi_H \colon \C_G^{[f]}(A)\to \C_H^{[f]}(A),$ sending $X_i^{G}\mapsto X_i^{H}$,
is induced by the natural projection map
$\pi_H \colon \Phi_G(A) \to \Phi_H(A)$. By the definition of $\Phi_G$
as a square-free algebra on the set of edges,
the map $\pi_H$ splits, i.e., $\Phi_G(A) = \Phi_H(A) \oplus e\Phi_H(A)$.
The functoriality of the definitions gives a commutative diagram
\[\begin{tikzcd}
 \C_G^{[f]}\arrow{r}{\pi_H} \arrow[hookrightarrow, swap]{d}{} & \C_H^{[f]}
 \arrow[hookrightarrow]{d}{} \arrow{r}{}  & 0  \\
\Phi_G \arrow{r}{\pi_H} & \Phi_H \arrow{r}{}  & 0.
\end{tikzcd}
\]
Therefore, Lemma~\ref{lm: direct criterion} asserts that
$\C_G^{[f]}(A) = e\Phi_H(A) \cap \C_G^{[f]}(A) \oplus \Phi_H(A)\cap \C_G^{[f]}(A)$.
Using Lemma~\ref{lm: direct criterion}
again, we extend the splitting to the filtrations;  note that $\pi_H$ respects the filtration, i.e.,
\[
  \pi_H \left (\langle \prod f(X_i)^{a_i} \mid \sum a_i \leq n \rangle \right)
  \subseteq \langle \prod (\pi_H f(X_i))^{a_i} \mid \sum a_i \leq n \rangle
  \subseteq  \C_H^{[f]}(A).
\]
\end{proof}

\begin{theorem} \label{th:extra}
Suppose that $\K$ is algebraically closed. Then 
the projection map $\pi_H$ induces the surjective map $\cA_G\to \cA_H$
which preserves the Hilbert stratifications, i.e., the image of a connected component
of a Hilbert stratum is still contained in one stratum.
\end{theorem}
\begin{proof}
As in the proof of Corollary~\ref{cor constructive} we will parametrize $\cA_G$
as the fibers  of $A := \K[T_2, \ldots, T_v] \to \C_G^f(A)$.

By Corollary~\ref{cor constructive}, it suffices to show that if
$\mf p \in \Spec A$ and $s \notin \mf p$ are such that $V_{\K}(\mf p) \cap D(s)$
is contained in a single Hilbert stratum of $\C_G^f(A)$, then it is contained in
a single Hilbert stratum of $\C_H^f(A)$. Here, $V_{\K}(\mf p)$ denotes the set of 
$\K$-rational points containing $\mf p$ and $D(s)$ denotes the distinguished
open set of points not containing $s$.
By the construction we can pass to $B := A_s/\mf pA_s$ and consider the
stratifications induced by $B$-algebras $\C_G^f(B)$ and its filtered direct
summand $\C_H^f(B)$, see Theorem~\ref{thm graph splitting}.

Because $\K$ is algebraically closed, the set of $\K$-rational points is
dense, so the Hilbert stratification of $G$ is constant on the entire $\Spec B$.
Thus, by Theorem~\ref{thm semi} and Corollary~\ref{cor constructive} the
condition on the stratification can be restated as freeness  of $\C_G^f(B) \otimes_B B_\mf q$ (as a filtered
$B_\mf q$-module) for every $\mf q$-rational
point of $\Spec B_\mf q$.
But a direct summand of a free module is projective and projective modules over
a local ring are free.
\end{proof}

\begin{remark}
Like other similar results, Theorem~\ref{th:extra} shows that when $\K$ is not
algebraically closed, it is better to consider the stratification of the entire
space $\Spec A$ rather than only of the set of its $\K$-points.
\end{remark}

\begin{corollary}
For any usual graph $G$ on $n$ vertices, the Hilbert stratification of $\cA_G$  is
a coarsening of the Hilbert stratification of $\cA_{K_n}$, where $K_n$ denotes a complete graph on $n$ vertices. 
\end{corollary}

\section{Experimental results}  \label{sec:exp}

In this section,  we present the results of computations of the Hilbert
sequences $\hs_G^f$ of deformed zonotopal algebras $\CC^f_G$ for various graphs
$G$ and nondegenerate polynomials $f$  using Macaulay2 (\cite{M2}). Notice that, when
$f=u$, i.e.\ in the graded case, the answer is provided by
Theorem~\ref{th:forests}. In particular, if $G$ is a tree with $n$ edges, then
the Hilbert series of $\CC_G^u$ is equal to $(1+t)^n$.
For this reason, listing results of our computations below, we usually exclude
the graded case $f=u$. Also we do not include in $f$ monomials of degree higher
than $\md_G$, since as explained in Remark~\ref{rem:nilpot}, they do not affect
$\CC_G$.

\subsection{Special families of graphs}

In the tables below we present  the Hilbert
sequences for several families of graphs with $n=|V|$ vertices. Each row always
starts  with $1$, and the second entry for all $f\ne u$  equals $n$.

\begin{example}
\label{ex:An} For the chain graph with $n\ge 3$ vertices, $A_n=$
\begin{minipage}{3cm}
\begin{tikzpicture}
[scale=.5,auto=left,every node/.style={circle,fill=black!75,inner sep=0pt, minimum size=.15cm}]
  \node (n1) at (0,0) {};
  \node (n2) at (1,0) {};
  \node[fill=none] (n3) at (2.4,0) {\ $\cdots$};
  \node (n4) at (3.7,0) {};
   \node (n5) at (4.7,0) {};
  \foreach \from/\to in {n1/n2,n2/n3,n3/n4,n4/n5}
    \draw (\from) -- (\to);
  \node[right,fill=none] at (n5.east) {,};
\end{tikzpicture}
\end{minipage}
\
besides the graded case, $f=u$, we only need
to consider $f=u+u^2.$
The corresponding \HS sequences for $3\leq n\leq 17$, are listed below.
\smallskip

\begin{minipage}{5cm}
\begin{tabular}{rrrrrrrrrr}
1&3&&&&&&&&\\
1 &4&3&&&&&&\\
1& 5& 10&&&&&&\\
1& 6& 16& 9&&&&&\\
1& 7& 23& 33&&&&&\\
1& 8& 31& 61& 27&&&&&\\
1& 9& 40& 98& 108&&&&&\\
1& 10& 50& 145& 225& 81&&&&\\
1& 11& 61& 203& 397& 351&&&&\\
1& 12&73&273&636&810&243&&\\
1& 13&86&356&955&1551&1134&&&\\
1&14&100&453&1368&2665&2862&729&\\
1& 15&115&565&1890&4258&5895&3645&\\
1& 16& 131& 693& 2537& 6452& 10788& 9963& 2187\\
1& 17& 148& 838& 3326& 9386& 18232& 21924& 11664\\
\end{tabular}
\end{minipage}

\smallskip\noindent

Conjectures:
\begin{itemize}
\item[-]
 for $n\ge 5$, the $3$rd entry equals\;
$\ds\frac{n^2+n-10}{2}$;
\item[-]
for $n\ge 7$,  the $4$th entry equals\;
$\ds \frac{n^3 +3n^2-46n+30}{6}$;
\item[-]  for $n\ge 9$, the $5$th entry equals\;
$\ds \frac{n^4+6n^3-121n^2+42n+1080}{24}$;
\item[-]   for $n\ge 11$, the $6$th entry equals\;
$\ds \frac{n^5+10n^4-245n^3-250n^2+9364n-12000}{120}$;
\item[-]  for $n\ge 13$, the $7$th entry equals\;
$\ds \frac{n^6+15n^5+3325n^4+1785n^3+11874n^2-201960n+93600}{6!}$;
\item[-] for $n=2k$, the $(k+1)$st entry equals  $3^{k-1}$.
\end{itemize}
\end{example}

\begin{example}
For the cycle graph with $n$ vertices, $\widehat{A}_n=$
\hfill
\begin{minipage}{3cm}
\begin{tikzpicture}
  [scale=.4,auto=left,every node/.style={circle,fill=black!75,inner sep=0pt,
    minimum size=.15cm}]
 \node (n1) at (0,-4) {};
  \node (n2) at (-1,-2) {};
  \node (n3) at (0,0) {};
  \node[fill=none] (n4) at (2,1) {$\substack{\ \\ \ \cdots \  }$};
  \node (n5) at (4,0) {};
  \node (n6) at (5,-2) {};
  \node (n7) at (4,-4) {};
   \node[fill=none] (n8) at (2,-5) {$\substack{\ \cdots \ \\ \ }$};
  \foreach \from/\to in {n1/n2,n2/n3,n3/n4,n4/n5,n5/n6,n6/n7, n7/n8,    n8/n1}
    \draw (\from) -- (\to);
  \node[right,fill=none] at (n6.east) {,};
\end{tikzpicture}
\end{minipage}
besides the graded case, $f=u$, we only need  to consider $f=u+u^2.$
The corresponding \HS sequences, for $3\leq n \leq 13$ are listed
below.

\smallskip
\begin{minipage}{4cm}
\begin{tabular}{rrrrrrrrrr}
1& 3& 2& 1&&&&&\\
1& 4& 7& 3&&&&&\\
1& 5& 14& 10& 1&&&&\\
1& 6& 20& 31& 5&&&&\\
1& 7& 27& 63& 28& 1&&&&\\
1& 8& 35& 96& 106&9&&&&\\
1& 9& 44&138& 243& 75& 1&&&\\
1&10& 54&190&405&346& 17&&&\\
1& 11& 65& 253& 627& 891& 198& 1&&\\
1&12&77&328&921&1620&1103&33&\\
1&13& 90&416&1300&2691&3159&520&1\\
\end{tabular}
\end{minipage}

\smallskip\noindent
Conjectures:
\begin{itemize}
\item[-]
for $n\ge 5$, the $3$rd entry equals $\frac{n^2+n-2}{2}$,
\item[-] for $n\ge 7$,  the $4$th entry equals  $\frac{n^3 +3n^2-16n}{6}$. 
\end{itemize}
\end{example}

\begin{example}
\label{ex:Dn}
For  the graph $D_n=$ \
\begin{minipage}[c]{3.2cm}
\begin{tikzpicture}
  [scale=.4, auto=left, every node/.style={circle, fill=black!75,
    inner sep=0pt, minimum size=.15cm}]
  \node (n1) at (0,0) {};
  \node (n2) at (1,0) {};
  \node[fill=none] (n3) at (3,0) {$\cdots$};
  \node (n4) at (5,0) {};
   \node (n5) at (6,0) {};
      \node (n6) at (7,1) {};
         \node (n7) at (7,-1) {};
  \foreach \from/\to in {n1/n2,n2/n3,n3/n4,n4/n5, n5/n6, n5/n7}
    \draw (\from) -- (\to);
\end{tikzpicture}
\end{minipage}
with $n\geq 4$ vertices
and  $f=u+u^3$
(for $f=u+u^2$ and $f=u+u^2+u^3$ the Hilbert sequences are the same
as for the graph $A_n$, see Example~\ref{ex:An}), we have the
following Hilbert sequences

 \smallskip
 \begin{minipage}[c]{4cm}
  \begin{tabular}{rrrrrrrrrrr}
1& 4& 3& &&&&&&&\\
1& 5& 7& 3&&&&&&&\\
1& 6& 12& 10& 3&&&&&&\\
1& 7& 18& 22& 13& 3&&&&&\\
1& 8& 25& 40& 35& 16& 3&&&&\\
1& 9& 33& 65& 75& 51& 19& 3&&&\\
1& 10& 42& 98& 140& 126& 70& 22& 3&\\
1& 11& 52& 140& 238& 266& 196& 92& 25& 3\\
\end{tabular}
 \end{minipage}

\smallskip\noindent
Conjecture. These numbers satisfy a Pascal-type recursion relation.
\end{example}

\begin{example}
\label{ex:An3}
For the graph $G=$
 \
\begin{minipage}{3cm}
\begin{tikzpicture}  [scale=.4, auto=left, every node/.style={circle, fill=black!75,
inner sep=0pt, minimum size=.15cm}]
  \node (n1) at (0,0) {};
  \node (n2) at (1,0) {};
 \node[fill=none] (n3) at (3,0) {$\cdots$};
 \node (n4) at (5,0) {};
 \node (n5) at (6,0) {};
 \node (n6) at (7,0) {};
 \node (n8) at (5,1) {};
  \foreach \from/\to in {n1/n2,n2/n3,n3/n4,n4/n5, n5/n6, n4/n8}
    \draw (\from) -- (\to);
\end{tikzpicture}
\end{minipage}
\ with $n\ge 5$ vertices and $f=u+u^2$
(for $f=u+u^3$ the Hilbert sequences are the same as for the graph
$D_n$ from Example~\ref{ex:Dn}), we have the following Hilbert sequences

\smallskip
 \begin{tabular}{rrrrrrr}
1& 5& 10&&&&\\
1& 6& 15& 10&&&\\
1& 7& 22& 34&&&\\
1& 8& 30& 59& 30&&\\
1& 9& 39& 95& 112&\\
1& 10& 49& 141& 221& 90\\
1& 11& 60& 198& 388& 366\\
\end{tabular}
\end{example}

\begin{example}
\label{ex:An4} For $G=$
\begin{minipage}{3.5cm}
\begin{tikzpicture}
[scale=.4,auto=left,every node/.style={circle, fill=black!75, inner sep=0pt, minimum
  size=.15cm}]
  \node (n1) at (0,0) {};
  \node (n2) at (1,0) {};
\node[fill=none] (n3) at (3,0) {$\cdots$};
  \node (n4) at (5,0) {};
   \node (n5) at (6,0) {};
    \node (n6) at (7,0) {};
    \node (n7) at (8,0) {};
      \node (n8) at (5,1) {};
  \foreach \from/\to in {n1/n2,n2/n3,n3/n4,n4/n5, n5/n6, n6/n7, n4/n8}
    \draw (\from) -- (\to);
\end{tikzpicture}
\end{minipage}
($A_{n-1}$ with a leg on the 4th place), with $n\geq 7$ vertices
and $f=u+u^2$
 (for $f=u+u^3$ and $f=u+u^2+u^3$,  the results are the same as for
 the graphs $D_n$
 from Example~\ref{ex:Dn} and $A_n$, from Example~\ref{ex:An},
 respectively), we have the Hilbert sequences

\smallskip
\begin{minipage}{3cm}
 \begin{tabular}{rrrrrrr}
1& 7& 22& 34&&\\
1& 8& 30& 62& 27&\\
1& 9& 39& 96& 111&\\
1& 10& 49& 142& 229& 81\\
1& 11& 60& 199& 393& 360\\
\end{tabular}
\end{minipage}
\end{example}

\begin{example}
For  $G=$
\begin{minipage}{4cm}
\begin{tikzpicture}
[scale=.4,auto=left,every node/.style={circle,fill=black!75, inner sep=0pt, minimum
  size=.15cm}]
  \node (n1) at (0,0) {};
  \node (n2) at (1,0) {};
  \node[fill=none] (n3) at (3,0) {$\cdots$};
  \node (n4) at (5,0) {};
   \node (n5) at (6,0) {};
    \node (n6) at (7,0) {};
    \node (n7) at (8,0) {};
          \node (n8) at (9,0) {};
      \node (n9) at (5,1) {};
  \foreach \from/\to in {n1/n2,n2/n3,n3/n4,n4/n5, n5/n6, n6/n7, n7/n8, n4/n9}
    \draw (\from) -- (\to);
\end{tikzpicture}
\end{minipage}
($A_{n-1}$ with a leg on $5$th place) with $n\geq 6$ vertices
and $f=u+u^2$ (for $f=u+u^3$ and $f(u)=u+u^2+u^3$, the results are the
same as for graphs $D_n$ and $A_n$, respectively), we have the Hilbert 
sequences

\smallskip
\begin{minipage}{3cm}
\begin{tabular}{rrrrrrr}
1& 6& 16& 9&&&\\
1& 7& 23& 33&&&\\
1& 8& 30& 59& 30&&\\
1& 9& 39& 96& 111&  \\
1& 10& 49& 140& 222& 90\\
1& 11& 60& 197& 392&363\\
\end{tabular}
\end{minipage}
 \hfill
\end{example}

\begin{example}
For graph  $\widehat{D}_n=$
\
\begin{minipage}{3.5cm}
\begin{tikzpicture}
[scale=.4,auto=left,every node/.style={circle,fill=black!75,
    inner sep=0pt, minimum size=.15cm}]
  \node (n8) at (-1,-1) {};
  \node (n9) at (-1,1) {};
  \node (n1) at (0,0) {};
  \node (n2) at (1,0) {};
  \node[fill=none] (n3) at (3,0) {$\ \cdots$};
  \node (n4) at (5,0) {};
  \node (n5) at (6,0) {};
  \node (n6) at (7,1) {};
  \node (n7) at (7,-1) {};
\foreach \from/\to in {n1/n2,n2/n3,n3/n4,n4/n5, n5/n6, n5/n7, n1/n8, n1/n9}
    \draw (\from) -- (\to);
\end{tikzpicture}
\end{minipage}
\
and $f=u+u^3$  (for $f=u+u^2$ and $f=u+u^2+u^3$,
the results are the same as for $A_n$), we have the Hilbert sequences

\smallskip
 \begin{tabular}{rrrrrrrrrrr}
1& 5& 7& 3&&&&&&\\
1& 6& 14& 10& 1&&&&&\\
1& 7& 22& 25& 9&&&&&\\
1& 8& 30& 47& 33& 9&&&&\\
1& 9& 39& 77& 79& 42& 9&&&\\
1& 10& 49& 116& 155& 121& 51& 9&&\\
1& 11& 60& 165& 270& 276& 172& 60& 9&\\
1& 12& 72& 225& 434& 546& 448& 232& 69& 9\\
\end{tabular}

 \smallskip\noindent
Conjecture. These numbers exhibit a Pascal-type behavior.
\end{example}

\begin{example}
For $G=$
\
\begin{minipage}{3.6cm}
\begin{tikzpicture}
[scale=.4,auto=left,every node/.style={circle,fill=black!75, inner sep=0pt, minimum
  size=.15cm}]
  \node (n0) at (-1,0) {};
  \node (n1) at (0,0) {};
  \node (n2) at (1,0) {};
\node[fill=none] (n3) at (3,0) {$\cdots$};
  \node (n4) at (5,0) {};
   \node (n5) at (6,0) {};
      \node (n6) at (7,1) {};
      \node (n7) at (7,-1) {};
       \node (n8) at (1,1) {};
\foreach \from/\to in {n0/n1, n1/n2,n2/n3,n3/n4,n4/n5, n5/n6, n5/n7,n2/n8}
    \draw (\from) -- (\to);
\end{tikzpicture}
\end{minipage}
and $f=u+u^3$ (for $f=u+u^2$ and $f=u+u^2+u^3$ the answers are the
same as in Example~\ref{ex:An3} and Example~\ref{ex:An},
respectively), we have the Hilbert sequences

\smallskip
 \begin{tabular}{rrrrrrrr}
1& 7& 20& 24& 11& 1&\\
1& 8& 29& 47& 34& 9&\\
1& 9& 38& 77& 80& 42& 9\\
\end{tabular}

\end{example}

\begin{example}
For $G=$ \
\begin{minipage}{3.2cm}
\begin{tikzpicture}
[scale=.4,auto=left,every node/.style={circle,fill=black!75, inner sep=0pt,
  minimum size=.15cm}]
  \node (n1) at (0,0) {};
  \node (n2) at (1,0) {};
\node[fill=none] (n3) at (3,0) {$\cdots$};
  \node (n4) at (5,0) {};
   \node (n5) at (6,0) {};
      \node (n6) at (7,1) {};
      \node (n7) at (7,-1) {};
       \node (n8) at (5,1) {};
  \foreach \from/\to in {n1/n2,n2/n3,n3/n4,n4/n5, n5/n6, n5/n7,n4/n8}
    \draw (\from) -- (\to);
\end{tikzpicture}
\end{minipage}
\
and
 $f=u+u^3$  (for $f=u+u^2$ and $f=u+u^2+u^3$ the results
 are the same as in Example~\ref{ex:An4} and Example~\ref{ex:An},
 respectively), we have the Hilbert sequences
\smallskip

\begin{minipage}{4cm}
 \begin{tabular}{rrrrrrrrrr}
1& 5& 7& 3&&&&&\\
1& 6& 14& 10& 1&&&&\\
1& 7& 20& 24& 11& 1&&&\\
1& 8& 27& 44& 35& 12& 1&&\\
1& 9& 35& 71& 79& 47& 13& 1&\\
1& 10& 44& 106& 150& 126& 60& 14& 1\\
\end{tabular}
\end{minipage}

\smallskip\noindent
Conjecture: Pascal-type behavior except for the entry $14$.
\end{example}

\begin{example}
For $G=$
\
\begin{minipage}{4cm}
\begin{tikzpicture}
[scale=.4,auto=left,every node/.style={circle,fill=black!75, inner sep=0pt, minimum
  size=.15cm}]
  \node (n1) at (0,0) {};
  \node (n2) at (1,0) {};
\node[fill=none] (n3) at (3,0) {$\cdots$};
  \node (n4) at (5,0) {};
  \node (n5) at (5,1) {};
   \node (n6) at (6,0) {};
      \node (n7) at (7,0) {};
      \node (n8) at (8,-1) {};
      \node (n9) at (8,1) {};
  \foreach \from/\to in {n1/n2,n2/n3,n3/n4,n4/n5, n4/n6, n6/n7, n7/n8, n7/n9}
    \draw (\from) -- (\to);
\end{tikzpicture}
\end{minipage}
and $f=u+u^3$, we have the Hilbert sequences
\smallskip

\begin{minipage}{4cm}
 \begin{tabular}{rrrrrrrrr}
1& 7& 22& 25& 9&&&\\
1& 8& 29& 47& 34& 9&&\\
1&9& 37& 76& 81& 43& 9&\\
1& 10& 46& 113& 157& 124& 52& 9\\
\end{tabular}
\end{minipage}

\smallskip\noindent
Conjecture: A Pascal-type behavior.
\end{example}

\begin{example}
For  $G=$
\
\begin{minipage}{4.1cm}
\begin{tikzpicture}
  [scale=.4,auto=left,every node/.style={circle,fill=black!75,
    inner sep=0pt, minimum size=.15cm}]
  \node (n1) at (0,0) {};
  \node (n2) at (1,0) {};
\node[fill=none] (n3) at (3,0) {$\cdots$};
  \node (n4) at (5,0) {};
  \node (n5) at (5,1) {};
   \node (n6) at (6,0) {};
      \node (n7) at (7,0) {};
      \node (n8) at (8,0) {};
     \node (n9) at (9,1) {};
      \node (n10) at (9,-1) {};
  \foreach \from/\to in {n1/n2,n2/n3,n3/n4,n4/n5, n4/n6, n6/n7, n7/n8, n8/n9, n8/n10}
    \draw (\from) -- (\to);
\end{tikzpicture}
\end{minipage}
 and $f=u+u^3$, we have the Hilbert sequences
\smallskip

\begin{minipage}{4cm}
 \begin{tabular}{rrrrrrrrr}
1& 8& 30& 47& 33& 9&&\\
1& 9& 38& 77& 80& 42& 9,&\\
1& 10& 47& 115& 157& 122& 51& 9\\
\end{tabular}
\end{minipage}

\smallskip\noindent
Conjecture: A Pascal-type behavior.
\end{example}

\subsection{Trees of maximal degree at most three }

\subsubsection{$4$ vertices}

The homogeneous (i.e. corresponding to $f(u)=u$) \HS sequence is $(1,3,3,1)$.
\renewcommand{\labelitemi}{$\diamond$}
\begin{itemize}
\item
\begin{tikzpicture}
[scale=.4,auto=left,every node/.style={circle,fill=black!75,inner sep=0pt, minimum size=.15cm}]
  \node (n1) at (0,0) {};
  \node (n2) at (1,0) {};
  \node (n3) at (2,0) {};
  \node (n4) at (3,0) {};
  \foreach \from/\to in {n1/n2,n2/n3,n3/n4}
    \draw (\from) -- (\to);
\end{tikzpicture}
has the \HS sequence  $(1,4,3)$ 
 for all $f(u) = u + au^2$ with $a \neq 0$.
\item
\begin{tikzpicture}
[scale=.4,auto=left,every node/.style={circle,fill=black!75,inner sep=0pt, minimum size=.15cm}]
  \node (n1) at (0,0) {};
  \node (n2) at (1,0) {};
  \node (n3) at (1, 1) {};
  \node (n4) at (2,0) {};
  \foreach \from/\to in {n1/n2,n2/n3,n2/n4}
    \draw (\from) -- (\to);
\end{tikzpicture}
has the \HS sequence $(1,4,3)$ 
   for all $f(u) = u + au^2 + bu^3$ unless $a = b = 0$.
\end{itemize}

\subsubsection{$5$ vertices}

The homogeneous \HS sequence  is  $(1,4,6,4,1)$. 
\begin{itemize}
\item
\begin{tikzpicture}
[scale=.4,auto=left,every node/.style={circle,fill=black!75,inner sep=0pt, minimum size=.15cm}]
  \node (n1) at (0,0) {};
  \node (n2) at (1,0) {};
  \node (n3) at (2,0) {};
  \node (n4) at (3,0) {};
  \node (n5) at (4,0) {};
  \foreach \from/\to in {n1/n2,n2/n3,n3/n4,n4/n5}
    \draw (\from) -- (\to);
\end{tikzpicture}
has the \HS sequence $(1,5,10)$  
 for all $f(u) =  u+ au^2$ with $a \neq 0$.
\item
\begin{tikzpicture}
[scale=.4,auto=left,every node/.style={circle,fill=black!75,inner sep=0pt, minimum size=.15cm}]
  \node (n1) at (0,0) {};
  \node (n2) at (1,0) {};
  \node (n3) at (2, 0) {};
  \node (n4) at (2,1) {};
   \node (n5) at (3,0) {};
  \foreach \from/\to in {n1/n2,n2/n3,n3/n4,n3/n5}
    \draw (\from) -- (\to);
\end{tikzpicture}
has the \HS sequence $(1,5,10)$  
  for all $f(u) = u + au^2 + bu^3$ unless $a = 0$. For $f(u) = u + b u^3$ with  $b\neq 0,$
the \HS  sequence  is $(1,5,7,3)$.

\end{itemize}

\subsubsection{$6$ vertices}

The homogeneous \HS sequence  is $(1,5,10,10,5,1)$. 
Surprisingly all 4 non-isomorphic trees on $6$ vertices have the  same
general \HS sequence $(1,6,16, 9)$ 
 for $f(u) = u + au^2 + bu^3$ with $a, b \neq 0$.

\begin{itemize}
\item
\begin{tikzpicture}
[scale=.4,auto=left,every node/.style={circle,fill=black!75,inner sep=0pt, minimum size=.15cm}]
  \node (n1) at (0,0) {};
  \node (n2) at (1,0) {};
  \node (n3) at (2,0) {};
  \node (n4) at (3,0) {};
  \node (n5) at (4,0) {};
  \node (n6) at (5,0) {};
  \foreach \from/\to in {n1/n2,n2/n3,n3/n4,n4/n5,n5/n6}
    \draw (\from) -- (\to);
\end{tikzpicture}
has the \HS sequence $(1,6,16,9)$
   for all $f(u) = u + au^2$ with $a \neq 0$.
\item
\begin{minipage}[c]{2.2cm}
\begin{tikzpicture}
[scale=.4,auto=left,every node/.style={circle,fill=black!75,inner sep=0pt, minimum size=.15cm}]
  \node (n1) at (0,1) {};
  \node (n2) at (1,1) {};
  \node (n3) at (2,1) {};
  \node (n4) at (3,1) {};
  \node (n5) at (4,0) {};
  \node (n6) at (4,2) {};
  \foreach \from/\to in {n1/n2,n2/n3,n3/n4,n4/n5,n4/n6}
    \draw (\from) -- (\to);
\end{tikzpicture}
\end{minipage}
has special value at $a = 0$; the family $f(u) = u + bu^3,\; b\neq 0$ gives the \HS sequence $(1,6,12,10,3)$.
\item
\begin{minipage}[c]{2.2cm}
\begin{tikzpicture}
[scale=.4,auto=left,every node/.style={circle,fill=black!75,inner sep=0pt, minimum size=.15cm}]
  \node (n1) at (0,0) {};
  \node (n2) at (0,2) {};
  \node (n3) at (1,1) {};
  \node (n4) at (2,1) {};
  \node (n5) at (3,0) {};
  \node (n6) at (3,2) {};
  \foreach \from/\to in {n1/n3,n2/n3,n3/n4,n4/n5,n4/n6}
    \draw (\from) -- (\to);
\end{tikzpicture}
\end{minipage}
also has special value at $a = 0$: the family $f(u) = u + bu^3,\; b\neq 0$ gives the \HS sequence  $(1,6,14,10,1)$.
\item
\begin{tikzpicture}
[scale=.4,auto=left,every node/.style={circle,fill=black!75,inner sep=0pt, minimum size=.15cm}]
  \node (n1) at (0,0) {};
  \node (n2) at (1,0) {};
  \node (n3) at (2,0) {};
  \node (n4) at (2,1) {};
  \node (n5) at (3,0) {};
  \node (n6) at (4,0) {};
  \foreach \from/\to in {n1/n2,n2/n3,n3/n4,n3/n5,n5/n6}
    \draw (\from) -- (\to);
\end{tikzpicture}
has two special values: $a = 0, b\neq 0$ gives  the \HS sequence $(1,6,12,10,3)$ 
  while $a\neq 0, b = 0$ gives $(1,6,15,10)$. 
\end{itemize}

\subsubsection{$7$ vertices}

The homogeneous \HS  sequence  is $(1,6,15,20,15,6,1)$. The 
 general \HS  sequence is no longer shared by all trees with $7$
 vertices.

\begin{itemize}
\item
\begin{tikzpicture}
[scale=.4,auto=left,every node/.style={circle,fill=black!75,inner sep=0pt, minimum size=.15cm}]
  \node (n0) at (-1, 0) {};
  \node (n1) at (0,0) {};
  \node (n2) at (1,0) {};
  \node (n3) at (2,0) {};
  \node (n4) at (3,0) {};
  \node (n5) at (4,0) {};
  \node (n6) at (5,0) {};
  \foreach \from/\to in {n0/n1, n1/n2,n2/n3,n3/n4,n4/n5,n5/n6}
    \draw (\from) -- (\to);
\end{tikzpicture}
has the \HS  sequence $(1,7,23,33)$  
 for all $f(u) = u + au^2$ with $a \neq 0$.
\item
\begin{minipage}[c]{2.5cm}
\begin{tikzpicture}
[scale=.4,auto=left,every node/.style={circle,fill=black!75,inner sep=0pt, minimum size=.15cm}]
  \node (n0) at (-1,1) {};
  \node (n1) at (0,1) {};
  \node (n2) at (1,1) {};
  \node (n3) at (2,1) {};
  \node (n4) at (3,1) {};
  \node (n5) at (4,0) {};
  \node (n6) at (4,2) {};
  \foreach \from/\to in {n0/n1,n1/n2,n2/n3,n3/n4,n4/n5,n4/n6}
    \draw (\from) -- (\to);
\end{tikzpicture}
\end{minipage}
has the general 
\HS  sequence $(1,7,23,33)$ 
 and one special value $a=0$ which for
$f(u) = u + bu^3,\; b\neq 0$ gives the \HS sequence $(1,7,22, 25, 9)$.
\item
\begin{minipage}[c]{2.5cm}
\begin{tikzpicture}
[scale=.4,auto=left,every node/.style={circle,fill=black!75,inner sep=0pt, minimum size=.15cm}]
  \node (n1) at (0,0) {};
  \node (n2) at (0,2) {};
  \node (n3) at (1,1) {};
  \node (n4) at (2,1) {};
  \node (n5) at (3,1) {};
  \node (n6) at (4,0) {};
  \node (n7) at (4,2) {};
  \foreach \from/\to in {n1/n3,n2/n3,n3/n4,n4/n5,n5/n6, n5/n7}
    \draw (\from) -- (\to);
\end{tikzpicture}
\end{minipage}
has the
general
\HS  sequence $(1,7,23,33)$  
 and one special value $a=0$ which for
$f(u) = u + bu^3, b\neq 0$ gives the \HS  sequence $(1,7,22, 25, 9)$. 
\item
\begin{tikzpicture}
[scale=.4,auto=left,every node/.style={circle,fill=black!75,inner sep=0pt, minimum size=.15cm}]
  \node (n1) at (0,0) {};
  \node (n2) at (1,0) {};
  \node (n3) at (2,0) {};
  \node (n4) at (2,1) {};
  \node (n5) at (3,0) {};
  \node (n6) at (4,0) {};
   \node (n7) at (2,2){};
  \foreach \from/\to in {n1/n2,n2/n3,n3/n4,n3/n5,n5/n6,n4/n7}
    \draw (\from) -- (\to);
\end{tikzpicture}
has a different  general
\HS  sequence $(1,7,22,34)$ 
  and at least two special values $a=0$ and $b=0$. In particular,
for $f(u) = u + u^2$,  the \HS sequence  equals $(1,7,21,35)$ 
 and for $f(u) = u + u^3, $ the \HS sequence equals $(1,7,18,22,13,3)$.  
\item
\begin{minipage}[c]{2.5cm}
\begin{tikzpicture}
[scale=.4,auto=left,every node/.style={circle,fill=black!75,inner sep=0pt, minimum size=.15cm}]
  \node (n0) at (-1,1) {};
  \node (n1) at (0,1) {};
  \node (n2) at (1,1) {};
  \node (n3) at (1,0) {};
  \node (n4) at (2,1) {};
  \node (n5) at (3,0) {};
  \node (n6) at (3,2) {};
  \foreach \from/\to in {n0/n1,n1/n2,n2/n3,n2/n4,n4/n5,n4/n6}
    \draw (\from) -- (\to);
\end{tikzpicture}
\end{minipage}
has the
 general
\HS  sequence $(1,7,23,33)$ 
  and at least two special values $a=0$ and $b=0$. In particular,
for $f(u) = u + u^2$, the \HS sequence equals $(1,7,22,34)$ 
 and for $f(u) = u + u^3$,  the \HS sequence equals $(1,7,20,24,11,1)$. 
\item
\begin{tikzpicture}
[scale=.4,auto=left,every node/.style={circle,fill=black!75,inner sep=0pt, minimum size=.15cm}]
  \node (n0) at (-1,0) {};
  \node (n1) at (0,0) {};
  \node (n2) at (1,0) {};
  \node (n3) at (1,1) {};
  \node (n4) at (2,0) {};
  \node (n5) at (3,0) {};
  \node (n6) at (4,0) {};
  \foreach \from/\to in {n0/n1,n1/n2,n2/n3,n2/n4,n4/n5,n5/n6}
    \draw (\from) -- (\to);
\end{tikzpicture}
has the
 general
\HS  sequence  $(1,7,23,33)$ 
  and at least two special values $a=0$ and $b=0$.
For $f(u) = u + u^2$ the \HS sequence equals $(1,7,22,34)$ 
 and for $f(u) = u + u^3$ the \HS sequence equals $(1,7,18,22,13,3)$. 
\end{itemize}

\noindent
CAUTION: Parametric Gr{\" o}bner bases were not computed which means that
there could be missing additional special values of parameters which we have not found.

\subsubsection{$8$ vertices}

The homogeneous \HS  sequence is $(1,7,21,35,35,21,7,1)$.

\begin{itemize}
\item
\begin{tikzpicture}
[scale=.4,auto=left,every node/.style={circle,fill=black!75,inner sep=0pt, minimum size=.15cm}]
  \node (n0) at (-1, 0) {};
  \node (n1) at (0,0) {};
  \node (n2) at (1,0) {};
  \node (n3) at (2,0) {};
  \node (n4) at (3,0) {};
  \node (n5) at (4,0) {};
  \node (n6) at (5,0) {};
   \node (n7) at (6,0) {};
  \foreach \from/\to in {n0/n1, n1/n2,n2/n3,n3/n4,n4/n5,n5/n6,n6/n7}
    \draw (\from) -- (\to);
\end{tikzpicture}
For $f(u)=u + au^2,\, a\neq 0$ the \HS sequence equals $(1,8,31,61,27)$.
\item
\begin{minipage}[c]{2.7cm}
\begin{tikzpicture}
[scale=.4,auto=left,every node/.style={circle,fill=black!75,inner sep=0pt, minimum size=.15cm}]
  \node (n0) at (-1, 0) {};
  \node (n1) at (0,0) {};
  \node (n2) at (1,0) {};
  \node (n3) at (2,0) {};
  \node (n4) at (3,0) {};
  \node (n5) at (4,0) {};
  \node (n6) at (5,-1) {};
   \node (n7) at (5,1) {};
  \foreach \from/\to in {n0/n1, n1/n2,n2/n3,n3/n4,n4/n5,n5/n6,n5/n7}
    \draw (\from) -- (\to);
\end{tikzpicture}
\end{minipage}
For $f(u)=u + u^2$ and $f(u)=u + u^2 + u^3$,  the \HS sequence equals $(1,8,31,61,27)$. 
For $f(u)=u + u^3$,  it  equals $(1, 8, 25, 40,$ $ 35, 16, 3)$. 
\item
\begin{minipage}[c]{2.5cm}
\begin{tikzpicture}
[scale=.4,auto=left,every node/.style={circle,fill=black!75,inner sep=0pt, minimum size=.15cm}]
  \node (n0) at (-1, 0) {};
  \node (n1) at (0,0) {};
  \node (n2) at (1,0) {};
  \node (n3) at (2,0) {};
  \node (n4) at (2,-1) {};
  \node (n5) at (3,0) {};
  \node (n6) at (4,-1) {};
   \node (n7) at (4,1) {};
  \foreach \from/\to in {n0/n1, n1/n2,n2/n3,n3/n4,n3/n5,n5/n6,n5/n7}
    \draw (\from) -- (\to);
\end{tikzpicture}
\end{minipage}
For $f(u)=u + u^2$, the \HS sequence equals $(1,8,30,62,27)$;  
  for $f(u)=u + u^2 + u^3$, it equals $(1,8,31,61, 27)$, 
   and for $f(u)=u + u^3$, it equals $(1,8,27,44,35,12,1)$. 

\item
\begin{minipage}[c]{2.5cm}
\begin{tikzpicture}
[scale=.4,auto=left,every node/.style={circle,fill=black!75,inner sep=0pt, minimum size=.15cm}]
  \node (n0) at (-1, 0) {};
  \node (n1) at (0,0) {};
  \node (n2) at (1,0) {};
  \node (n3) at (1,-1) {};
  \node (n4) at (2,0) {};
  \node (n5) at (3,0) {};
  \node (n6) at (4,-1) {};
   \node (n7) at (4,1) {};
  \foreach \from/\to in {n0/n1, n1/n2,n2/n3,n2/n4,n4/n5,n5/n6,n5/n7}
    \draw (\from) -- (\to);
\end{tikzpicture}
\end{minipage}
For $f(u)=u + u^2$, the \HS sequence equals $(1,8,30, 59,30)$; 
  for $f(u)=u + u^2 + u^3$, it equals $(1,8,31,61,27)$, 
   and  for $f(u)=u + u^3$, it equals $(1,8,29,47,34,9)$.  

\item
\begin{minipage}[c]{2.5cm}
\begin{tikzpicture}
[scale=.4,auto=left,every node/.style={circle,fill=black!75,inner sep=0pt, minimum size=.15cm}]
  \node (n0) at (0, -1) {};
  \node (n1) at (0,1) {};
  \node (n2) at (1,0) {};
  \node (n3) at (2,0) {};
  \node (n4) at (3,0) {};
  \node (n5) at (4,0) {};
  \node (n6) at (5,-1) {};
   \node (n7) at (5,1) {};
  \foreach \from/\to in {n0/n2, n1/n2,n2/n3,n3/n4,n4/n5,n5/n6,n5/n7}
    \draw (\from) -- (\to);
\end{tikzpicture}
\end{minipage}
For $f(u)=u + u^2$ and $f(u)=u + u^2 + u^3$, the \HS sequence  equals $(1,8,31,61,27)$;  
  for $f(u)=u + u^3$, it equals   $(1,8,30, 47, 33, 9)$. 

\item
\begin{minipage}[c]{2.5cm}
\begin{tikzpicture}
[scale=.4,auto=left,every node/.style={circle,fill=black!75,inner sep=0pt, minimum size=.15cm}]
  \node (n0) at (0, -1) {};
  \node (n1) at (0,1) {};
  \node (n2) at (1,0) {};
  \node (n3) at (2,0) {};
  \node (n4) at (2,-1) {};
  \node (n5) at (3,0) {};
  \node (n6) at (4,-1) {};
   \node (n7) at (4,1) {};
  \foreach \from/\to in {n0/n2, n1/n2,n2/n3,n3/n4,n3/n5,n5/n6,n5/n7}
    \draw (\from) -- (\to);
\end{tikzpicture}
\end{minipage}
For $f(u)=u + u^2$,  the \HS sequence equals $(1,8,30,62,27)$; 
  for $f(u)=u + u^2 + u^3$, it equals  $(1,8,31,61,27)$, 
  and  for $f(u)=u + u^3$, it equals $(1,8,29,51,34,5)$. 
\item
\begin{minipage}[c]{2.8cm}
\ \\[7pt]
\begin{tikzpicture}
[scale=.4,auto=left,every node/.style={circle,fill=black!75,inner sep=0pt, minimum size=.15cm}]
  \node (n0) at (-1, 0) {};
  \node (n1) at (0,0) {};
  \node (n2) at (1,0) {};
  \node (n3) at (2,0) {};
  \node (n4) at (3,0) {};
  \node (n5) at (3,-1) {};
  \node (n6) at (4,0) {};
   \node (n7) at (5,0) {};
  \foreach \from/\to in {n0/n1, n1/n2,n2/n3,n3/n4,n4/n5,n4/n6,n6/n7}
    \draw (\from) -- (\to);
\end{tikzpicture}
\end{minipage}
For $f(u)=u + u^2$, the \HS sequence equals  $(1,8,30, 59,30)$; 
  for $f(u)=u + u^2 + u^3$, it equals $(1,8,31,61,27)$, 
  and for $f(u)=u + u^3$, it equals $(1,8, 25, 40, 35,16,3)$.

\item
\begin{minipage}[c]{2.8cm}
\ \\[10pt]
\begin{tikzpicture}
[scale=.4,auto=left,every node/.style={circle,fill=black!75,inner sep=0pt, minimum size=.15cm}]
  \node (n1) at (0,0) {};
  \node (n2) at (1,0) {};
  \node (n3) at (2,0) {};
  \node (n4) at (3,0) {};
  \node (n5) at (3,-1) {};
  \node (n6) at (4,0) {};
   \node (n7) at (5,0) {};
  \node (n8) at (6,0){};
  \foreach \from/\to in {n1/n2,n2/n3,n3/n4,n4/n5,n4/n6,n6/n7, n7/n8}
    \draw (\from) -- (\to);
\end{tikzpicture}
\end{minipage}
For $f(u)=u + u^2$,  the \HS sequence equals $(1,8,30,62,27)$; 
  for $f(u)=u + u^2 + u^3$, it equals $(1,8,31,61,27)$, 
   and for $f(u)=u+u^3$, it equals  $(1,8,25,40,35,16,3)$.  
\item
\begin{minipage}[c]{2.8cm}
\ \\[10pt]
\begin{tikzpicture}
[scale=.4,auto=left,every node/.style={circle,fill=black!75,inner sep=0pt, minimum size=.15cm}]
  \node (n1) at (0,0) {};
  \node (n2) at (1,0) {};
  \node (n3) at (2,0) {};
  \node (n4) at (3,0) {};
  \node (n5) at (3,-1) {};
  \node (n6) at (3,-2) {};
   \node (n7) at (4,0) {};
  \node (n8) at (5,0){};
  \foreach \from/\to in {n1/n2,n2/n3,n3/n4,n4/n5,n5/n6,n4/n7, n7/n8}
    \draw (\from) -- (\to);
\end{tikzpicture}
\end{minipage}
For $f(u)=u + u^2$,  the \HS sequence equals $(1,8,29,60,30)$; 
 for $f(u)=u + u^2 + u^3$, it equals $(1,8,30,62,27)$,  
    and  for $f(u)=u + u^3$, it equals $(1,8,25, 40, 35,16,3)$.
\end{itemize}

\noindent
CAUTION: Parametric Gr{\" o}bner bases were not computed which means that
there could be missing additional special values of parameters which we have not found.


\subsection{Sporadic examples}

Here we present the \HS  sequences for several special graphs.

\smallskip
(A) We start with the complete graph $K_5$.

\begin{minipage}{7cm}
\begin{tabular}{l| rrrrrrrrrrr}
f(u) & 0 & 1 & 2 &3 & 4 & 5 & 6 & 7 & 8 & 9 & 10 \\ \hline
$u$ & 1 & 4 & 10 & 20 & 35 & 51 & 64 & 60 & 35 & 10 & 1 \\
$u + u^2$ & 1 & 5 & 14 & 30 & 55 & 80 & 77 & 15 & 9 & 4 & 1\\
$u + u^3$ & 1 & 5 & 15 & 33 & 60 & 76 & 60 & 27 & 9 & 4 & 1\\
$u + u^4$ & 1 & 5 & 14 & 30 & 53 & 73 & 60 & 41 & 9 & 4 & 1\\
$u + u^2 + u^3$ & 1 & 5 & 15 & 34 & 64 & 90 & 53 & 15 & 9 & 4 & 1\\
$u + u^2 + u^4$ & 1 & 5 & 15 & 35 & 67 & 91 & 48 & 15 & 9 & 4 & 1\\
$u + u^3 + u^4$ & 1 & 5 & 15 & 34 & 63 & 82 & 56 & 21 & 9 & 4 & 1\\
$u + u^2 + u^3 +u^4$ & 1 & 5 & 15 & 35 & 67 & 91 & 48 & 15 & 9 & 4 & 1\\
$u - \frac{u^2}{2} + \frac{u^3}{3} -\frac{u^4}{4}$ & 1 & 5 &15 & 35 & 63 & 84 & 59 & 15 & 9 &4 & 1
\end{tabular}
\end{minipage}

\smallskip
(B) $K_5\setminus e$, where $e$ is any edge.

\begin{tabular}{l| rrrrrrrrrr}
f(u) &${0}$& $1$ & $2$ & $3$ & $4$&$5$&$6$&$7$&$8$&$9$\\ \hline
$u$ & 1 & 4 & 10 & 20 & 33 & 45 & 46 & 29 & 9 & 1 \\
$u + u^2$ & 1 & 5 & 14 & 30 & 53 & 62 & 33 & && \\
$u + u^3$ & 1 & 5 & 15 & 33 & 55 & 59 & 28 & 2 && \\
$u + u^4$ & 1 & 5 & 14 & 30 & 48 & 50 & 37 & 13 &&\\
$u + u^2 + u^3$ & 1 & 5 & 15 & 34 & 62 & 68 & 13 &&&\\
$u + u^2 + u^4$ & 1 & 5 & 15 & 35 & 65 & 64 & 13 & &&\\
$u + u^3 + u^4$ & 1 & 5 & 15 & 35 & 65 & 64 & 13&&& \\
$u + u^2 + u^3 +u^4$ & 1 & 5 & 15 & 35 & 65 & 64 & 13 & &&\\
$u - \frac{u^2}{2} + \frac{u^3}{3} - \frac{u^4}{4}$ & 1 & 5 & 15 & 35 & 60 & 59 & 23 &&&
\end{tabular}

\smallskip

(C) $K_5\setminus (e_1\cup e_2)$, where $e_1$ and $e_2$ are any two disjoint edges of $K_5$.

\smallskip


\begin{tabular}{l| rrrrrrrrr}
f(u) & 0 & 1 & 2 &3 & 4 & 5 & 6 & 7 & 8  \\ \hline
$u$ & 1 & 4 & 10 & 20 & 31 & 35 & 24 & 8 & 1 \\
$u+ u^2$ & 1 & 5 & 14 & 30 & 51 & 26 & 7& &\\
$u + u^3$ & 1 & 5 & 15 & 33 & 45 & 29 & 6 & &\\
$u + u^4$ & 1 & 5 & 13 & 26 & 34 & 33 & 18 & 4 &\\
$u + u^2 + u^3$ & 1 & 5 & 15 & 34 & 56 & 19 & 4 &&\\
$u + u^2 + u^4$ & 1 & 5 & 15 & 34 & 56 & 19 & 4 &&\\
$u + u^3 + u^4$ & 1 & 5 & 15 & 34 & 47 & 28 & 4  &&\\
$u + u^2 + u^3 +u^4$ & 1 & 5 & 15 & 35 & 56 & 18 & 4 &  &\\
$u - \frac{u^2}{2} + \frac{u^3}{3} - \frac{u^4}{4}$ &1 & 5 & 15 & 35 & 51 & 20 & 7 &&
\end{tabular}


\smallskip


(D) A square with parallel double edges, \begin{minipage}{1cm}
\begin{tikzpicture}
  [scale=.5,auto=left,every node/.style={circle,fill=black!75,inner sep=0pt,
    minimum size=.15cm}]
 \node (n1) at (0,0) {};
  \node (n2) at (0,1) {};
  \node (n3) at (1,1) {};
  \node (n4) at (1,0) {};
 \draw[double] (n1) -- (n2){};
  \draw (n2) -- (n3){};
   \draw[double] (n3) -- (n4){};
    \draw (n4) -- (n1){};
\end{tikzpicture}
\end{minipage}

\smallskip

\begin{tabular}{l| rrrrrrr }
f(u) &  $0$& $1$ & $2$ & $3$ & $4$&$5$&$6$\\ \hline
$u$  &		1 & 3 & 6 & 9 & 8 & 4 & 1\\
$u + u^2$  &		1 & 4 & 9 & 15 & 3\\
$u + u^3$   &	 1 & 4 & 10 & 11 & 5 & 1 \\
$u + u^2 + u^3$   & 	1 & 4 & 10 & 15 & 2\\
$u - \frac{u^2}{2} + \frac{u^3}{3}$  & 1 & 4 & 10 & 14 & 3
\end{tabular}



\smallskip

(E) A square with adjacent double edges,
\begin{minipage}{1cm}
\begin{tikzpicture}
  [scale=.5,auto=left,every node/.style={circle,fill=black!75,inner sep=0pt,
    minimum size=.15cm}]
 \node (n1) at (0,0) {};
  \node (n2) at (0,1) {};
  \node (n3) at (1,1) {};
  \node (n4) at (1,0) {};
 \draw[double] (n1) -- (n2){};
  \draw (n2) -- (n3){};
   \draw (n3) -- (n4){};
    \draw[double] (n4) -- (n1){};
\end{tikzpicture}
\end{minipage}

The matroid is the same as in the previous example, so the original, undeformed algebra does not distinguish the two \cite{NSh}.

\smallskip

\begin{tabular}{l| rrrrrrr }
f(u)  &$0$ & $1$ & $2$ & $3$ & $4$&$5$&$6$ \\ \hline
$u$  &		1 & 3 & 6 & 9 & 8 & 4 & 1\\
$u + u^2$ & 1 & 4 &  9 & 15 & 3\\
$u + u^3$ &  1 & 4 & 10 & 12 & 5 \\
$u + u^4$  & 1 & 4 & 8 & 10 & 7 & 2\\
$u + u^2 + u^3$ &   	1 & 4 & 10 & 16 & 1\\
$u + u^2 + u^4$  & 1 & 4 & 10 & 16 & 1\\
$u + u^3 + u^4$ &    1 & 4 & 10 & 13 & 4\\
$u + u^2 + u^3 + u^4$   & 1 & 4 & 10 & 16 & 1\\
$u- \frac{u^2}{2} + \frac{u^3}{3}$ &   1 & 4 & 10 & 16 & 1\\
$u - \frac{u^2}{2} + \frac{u^3}{3} -\frac{u^4}{4}$ & 1 & 4 & 10 & 15 & 2
\end{tabular}



\subsection{Two examples with analysis of relations}

Here we present two small examples for which we can provide complete analysis of their weighted projective spaces  $P\cA_G$ of deformed zonotopal algebras.

\subsubsection{Multigraph $K_3+e$ with a double edge}

\begin{proposition}
 For $G=K_3 + e$,  the space $P\cA_G$ is stratified by the following functions over a field $\K$ of characteristic $0$.

 \smallskip
\begin{tabular}{l| rrrrr }
f(u) &  $0$ & $1$ & $2$ & $3$ & $4$\\ \hline
$u$ & 	$1$ & $2$ & $3$ & $3$ & $1$\\
$u + bu^2 + cu^3$ & $1$ & $3$ & $5$ & $1$ & \\
$*$ &	$1$ & $3$ & $4$ & $2$ &
\end{tabular}


 In the last  row either $b = 0$, or $c = 0$, or $3c = 4b^2$.
\end{proposition}

\begin{proof}
By Corollary~\ref{cor:relations} our relations have the  form
\[
\begin{cases}
x_1^4 = x_2^4 = x_3^3 = 0,\\
f(x_1) + f(x_2) + f(x_3) = 0,\\
(f(x_1) + f(x_2))^3 = (f(x_1) + f(x_3))^4 =  (f(x_2) + f(x_3))^4 = 0.
\end{cases}
\]
Note that, the first group of equations implies that $f(x_1)^4 = f(x_2)^4 = f(x_3)^3 = 0$,
so the third group of equation is redundant.

When $b \neq 0$ we may substitute $u \mapsto u/b$ to
transform $f(u)$ to the function of the form $f(u) = u + u^2 + \lambda u^3$.
Using $\lambda$ as a variable, we can verify
using Macaulay2 ((\cite{M2})) that $(x_1 + x_2 + x_3)^2$ is always a relation.
Note that since the sums of consecutive entries in any \HS  sequence are increasing,
the sequence is determined if there is no further relation in degree $2$. We want to
show that $\lambda = 0, 4/3$ are the only two exceptional cases where we get an
additional quadratic relation -- this will separate the two rows of the table.  It
can be verified by hand or using  Macaulay2, that \HS  sequences of the exceptional values of
$\lambda$ coincide.

By symmetry (if we had two relations which are switched by interchanging $x_1$ and $x_2$, then we add them),
we must have a quadratic relation of the form
\[
a_1(x_1^2 + x_2^2) + a_2x_1x_2 + a_3(x_1 + x_2) + bx_3(x_1 + x_2) + c_2x_3^2 + c_1x_3 = 0.
\]
By subtracting the existing relation $(x_1 + x_2 + x_3)^2 = 0$ we may assume that $b = 0$.
Since $f(x_3)^2 = x_3^2$ we may now rewrite the relation as
\[
a_1(x_1^2 + x_2^2) + a_2x_1x_2 + a_3(x_1 + x_2) = Q(-f(x_3)),
\]
where $Q$ is a quadratic polynomial such that $Q(0) = 0$. Since, $Q(-f(x_3)) = Q (f(x_1) + f(x_2))$ we need
to find when there is $Q$ such that $\deg Q(f(x_1) + f(x_2)) \leq 2$.
We may assume that $Q(T) = T^2 + aT$, because $Q(T) = T$ cannot work.

One can check that modulo existing relations $x_1^4 = x_2^4 = (f(x_1) + f(x_2))^3 = 0$, 
the polynomial $Q (f(x_1) + f(x_2))$ has degree at most $3$ and its cubic term is
\[
(-a\lambda + 4/3a) (x_1 + x_2)^3 + (\lambda + 4/3a) (x_1^3 + x_2^3).
\]
Both coefficients vanish if and only if
either $\lambda = 4/3, a = -1$ or $a = \lambda = 0$.
\end{proof}

\subsubsection{Graph $K_4$}\label{K4}

\begin{proposition}
For $G=K_4$, the space $\cA_G$  of parameters is stratified by the following functions with $b,c \neq 0$ over a field $\K$ of characteristic $0$.

\begin{table}[h!]
\begin{tabular}{l| rrrrrrr }
f(u)  & $0$ & $1$ & $2$ & $3$ & $4$&$5$&$6$\\ \hline
$u$ & 	$1$ & $3$ & $6$ & $10$ & $11$ & $6$ & $1$\\
$u + bu^2$ & $1$ & $4$ & $9$ & $15$ & $5$ & $3$ & $1$\\
$u + cu^3$ &  $1$ & $4$ &  $9$ & $12$ & $8$ & $3$ & $1$ \\
$u + bu^2 + cu^3$ & $1$ & $4$ & $10$ & $14$ & $5$ & $3$ & $1$
\end{tabular}
\end{table}

Thus we have a linear order of strata  of $\cA_G$  given by $u \prec u + c u^3 \prec u + b u^2 \prec u +b  u^2 + c u^3$.

\end{proposition}
\begin{proof}
We may use substitution $u \mapsto u/b$ to transform
the last equation into $u + u^2 + cu^3$.
Let us represent our algebra as a quotient of $K[x_1, x_2, x_3, x_4]$.

\smallskip
By Corollary~\ref{cor:relations} we have relations of the  form
\[
\begin{cases}
x_i^4 = 0 & \text{ for } i = 1, \ldots, 4\\
(x_i + x_j + x_i^2 + x_j^2 + c(x_i^3 + x_j^3))^5 = 0 & \text{ for } i \neq j\\
(x_i + x_j + x_k + x_i^2 + x_j^2 + x_k^2+ c(x_i^3 + x_j^3 + x_k^3))^4 = 0 & \text{ for } i \neq j \neq k\\
  x_1 + x_2 + x_3 + x_4 + (x_1^2+x_2^2+x_3^2+x_4^2) +
  c(x_1^3 + x_2^3 + x_3^3 & + x_4^3)=0.
\end{cases}
\]
It is then easy to see that
\[
(x_i + x_j + x_i^2 + x_j^2 + c(x_i^3 + x_j^3))^5 \equiv (x_i + x_j + x_i^2 + x_j^2)^5
\mod (x_1^4, x_2^4, x_3^4, x_4^4),
\]
In fact, this relation reduces to $40 x_i^3 x_j^3 + 10 x_i^3x_j^2 + 10 x_i^2 x_j^3
\equiv 0$. After multiplication by $x_i$ we get that $x_i^3x_j^3 \equiv 0
\mod (x_1^4, x_2^4, x_3^4, x_4^4)$.
So the ideal generated by the first two classes of relations is generated by
$x_i^4, x_i^3 x_j^2 + x_i^2 x_j^3$ for $i \neq j$.
It is now easy to verify that
\[
(x_i + x_j + x_k + x_i^2 + x_j^2 + x_k^2+ c(x_i^3 + x_j^3 + x_k^3))^4 \equiv (x_i +
x_j + x_k + x_i^2 + x_j^2 + x_k^2)^4
\]
modulo the ideal generated by the previous relations.
We conclude that the relations are, in fact, less dependent on $c$:
\[
\begin{cases}
x_i^4 = 0 & \text{ for } i = 1, \ldots, 4\\
(x_i + x_j + x_i^2 + x_j^2)^5 = 0 & \text{ for } i \neq j\\
(x_i + x_j + x_k + x_i^2 + x_j^2 + x_k^2)^4 = 0 & \text{ for } i \neq j \neq k\\
  x_1 + x_2 + x_3 + x_4 + (x_1^2+x_2^2+x_3^2+x_4^2) + c(x_1^3 + x_2^3 + x_3^3
  & + x_4^3)=0.
\end{cases}
\]
It is immediate that $c = 0$ gives a quadratic relation.
It is also easy to see that no other value of $c$ can give one.
Namely, since the ideal $I$ generated by the first three groups of relations,
which are independent of $c$, does not contain any element of order less than $4$,
then for any $f \in I,$ the polynomial
$$
f + x_1 + x_2 + x_3 + x_4 + (x_1^2+x_2^2+x_3^2+x_4^2) + c(x_1^3 + x_2^3 + x_3^3 +
x_4^3)
$$
has a cubic term.

\medskip
The case  $f(u) = u + cu^3$ is similar.
\end{proof}

\section {Outlook}\label{sec:out}
Here we present a small sample of  open problems about deformed zonotopal algebras $\CC_G^f$ for
future investigation.

\

Our experiments with Maculay2 show that for many graphs $G$ and functions $f$, the
Hilbert sequence of the algebra $\CC_G^f$ is logarithmically concave.
In an earlier preprint version of our paper we conjectured that this
was true for all graphs $G$ and all non-degenerate functions $f$.
However, in this generality the conjecture does not hold.  A
counterexample is given by the complete graph $G=K_5$ and $f=u+u^4$,
for which the Hilbert sequence is given by 
$\hs_G^f=(1, 5, 14, 30, 53, 73, 60, 41, 9, 4, 1)$ which is not
log-concave, since $9^2=81<41\cdot 4=164$.

\smallskip
However, recently a truly remarkable proof of this fact in the graded
case (i.e., $f=u$) was found in~\cite{EHL}. This circumstance gives hope that the
log-concavity might hold for a larger class of graphs and functions.

\begin{problem}
Find families of graphs $G$ and nondegenerate functions $f\ne u$ for
which the Hilbert sequence $\hs_G^f$ is log-concave. In particular,
prove log-concavity for chain and cycle graphs and $f=u+u^2$.
\end{problem}
\
\begin{problem}
When $G$ is a tree, is it true that the Hilbert stratification of
$\cA_G$ consists of coordinate subspaces?
In particular, does $f=e^u$ give a general  \HS sequence?
\end{problem}
\
\begin{problem} Find  a graph-theoretical interpretation of the
 general  \HS sequence $\hs_G$?
Find  a graph-theoretical interpretation of $\hs^f_{G}$
in  some special cases, for example, for $f=e^u$.
\end{problem}
\
\begin{problem} Is it true that if $\hs^f_{G_1}=\hs^f_{G_2}$ for every
function $f$, then $G_1$ is isomorphic to $G_2$?
  \end{problem}
\
\begin{problem}
Determine $G$ and $f$ for which the
algebra $\CC^f_G$ is Gorenstein/quadratic/Koszul.
\end{problem}
\

When $f=u$, i.e. in the graded case, the Hilbert sequences of $\CC_G$
satisfy the deletion-contraction relation~\eqref{eq:Tutte} which
allows to compute them recursively. However, for $f\ne u$ the
relation~\eqref{eq:Tutte} does not hold.

\begin{problem}
Study the behavior of $\hs^f_G$ under standard graph operations on
graphs, for example under deletions and contractions of edges.
\end{problem}

\end{document}